\documentclass{article}
\pdfoutput=1
\usepackage{verbatim}
\usepackage[utf8]{inputenc}
\usepackage{amsmath,amssymb,amsthm}
\usepackage{tikz-cd}
\usepackage{pgfplots}
\pgfplotsset{compat=1.11}
\usepgfplotslibrary{fillbetween}
\usetikzlibrary{intersections}
\usepackage{bbm}
\usepackage{stackengine,amsmath}
\stackMath
\usepackage{hyperref}
\usepackage{comment}
\usepackage{stmaryrd}
\usepackage[new]{old-arrows}
\usepackage{graphicx}
\graphicspath{ {./Images/} }
\usepackage{todonotes}
\usepackage{xcolor}
\usepackage{graphicx, graphics}
\usepackage [english]{babel}
\usepackage [autostyle, english = american]{csquotes}
\MakeOuterQuote{"}

\newcommand{\emb}[1]{\textcolor{black}{#1}}

\numberwithin{equation}{section}

\theoremstyle{definition}
\newtheorem*{definition}{Definition}
\newtheorem{remark}{Remark}

\newtheorem{claim}{Claim}
\newtheorem{propsec}{Proposition}[section]

\newtheorem*{main theorem}{Main Theorem}

\newtheorem{proposition}{Proposition}
\newtheorem{lemma}{Lemma}
\newtheorem{theorem}{Theorem}
\newtheorem{corollary}{Corollary}

\pgfdeclarelayer{bg}
\pgfsetlayers{bg,main}

\begin{document}

\title{On the Spectrum of Sturmian Hamiltonians\\ of Bounded Type in a Small Coupling Regime}
\author{\bf Alexandro Luna}
\date{}
\maketitle

\begin{abstract}
We prove that the Hausdorff dimension of the spectrum of a discrete Schr\"odinger operator with Sturmian potential of bounded type tends to one as coupling tends to zero. The proof is based on the trace map formalism.
\end{abstract}

\tableofcontents

\section{Introduction}
The object of study is the class of bounded self-adjoint discrete Schr\"odinger operators given by $H_{\lambda, \alpha, \omega}: \ell^2(\mathbb{Z})\rightarrow \ell^2(\mathbb{Z})$ via
$$\left[H_{\lambda,\alpha,\omega}u\right](n)= u(n+1)+u(n-1)+\left[\lambda\chi_{[1-\alpha, 1)}\circ R_{\alpha}^n(\omega)\right]u(n),$$
where $\lambda>0$, $\alpha$ is irrational, $\omega\in S^1$, \emb{where $S^1:=\mathbb R/\mathbb Z$ is the circle,} and $R_{\alpha}(\omega)=\omega+\alpha \ (\mathrm{mod} \ 1)$. We say that the potential term $v(n)=\lambda\chi_{[1-\alpha, 1)}\circ R_{\alpha}^n(\omega)$ is Sturmian, as it is generated by Sturmian subshifts (see \cite{D}), and refer to $\lambda$ as the coupling constant and $\alpha$ as the frequency. 

This operator falls under the category of a discrete Schr\"odinger operator with dynamically defined potential. This so-called Sturmian potential has been studied since the early 1980s, and has been a popular model in the study of one-dimensional quasicrystals. A Sturmian sequence is an aperiodic sequence of least possible complexity, and hence relates to the structure of quasicrystals which are highly ordered and aperiodic. A survey of results concerning these types of operators, especially Schr\"odinger operators, can be found in \emb{\cite{DEG,DF2}}.

We focus our attention to the case when the frequency has only finitely many different terms that appear in its continued fraction expansion (i.e. of bounded-type) and the coupling is weak (i.e. $\lambda$ is near $0$), and investigate the dimension of the spectrum of $H_{\lambda, \alpha, \omega}$ in this context. Since $R_{\alpha}$ is minimal, it is well-known (see Theorem 4.9 in \cite{DF}) that the spectrum of $H_{\lambda, \alpha, \omega}$ is independent of $\omega$, so we will denote it by $\sigma_{\lambda, \alpha}$. 

It is known from \cite{bist} that $\sigma_{\lambda, \alpha}$ is a Cantor set of zero Lebesgue measure. It was even recently proved in \cite{BBL} that all spectral gaps are open.
It is natural to ask about the fractal dimension of $\sigma_{\lambda, \alpha}$ and how it may depend on parameters $\lambda$ or $\alpha$. In application, the fractal dimension of $\sigma_{\lambda, \alpha}$ gives information about the solution to the time-dependent Sch\"odinger equation $i\partial_t\phi=H_{\lambda,\alpha,\omega}\phi$ (see \cite{JL, L}).

We prove that 
\begin{theorem}\label{main theorem}
    If $\alpha\in(0,1)$ is irrational of bounded-type, then 

    \begin{equation}\label{main limit}
        \lim\limits_{\lambda\rightarrow 0} \text{dim}_H\left(\sigma_{\lambda, \alpha}\right)=1.
    \end{equation}
    
\end{theorem}
This is a generalization of a result from \cite{M} where (\ref{main limit}) is proved under the assumption that $\alpha$ has an eventually periodic continued fraction expansion. An even earlier version of this result comes from \cite{dg3} where (\ref{main limit}) is proved for the case when $\alpha=\frac{\sqrt{5}-1}{2}$. The universal technique has been to estimate the \textit{thickness} of the Cantor set $\sigma_{\lambda, \alpha}$ using well-known theory from hyperbolic dynamics, but in our setting, we must introduce new methods to estimate the thickness.

In the case that $\alpha=\frac{\sqrt{5}-1}{2}$, the operator $H_{\lambda, \alpha, \omega}$ is known as the \textit{Fibonacci Hamiltonian}. It was first discovered that the spectrum of this operator is a Cantor set of zero Lebesgue measure in \cite{S89} and many questions regarding the box-counting and Hausdorff dimensions of the spectrum have been answered (see \cite{D00, D07, DEG, DGY, S95} for surveys). 

For general coupling constant, and when $\alpha=\frac{\sqrt{5}-1}{2}$, we know from a combination of results from \cite{C, Ca, DG1} that $\text{dim}_H\left(\sigma_{\lambda,\alpha}\right)$ is strictly between $0$ and $1$. If $\alpha$ has eventually periodic continued fraction expansion, then from \cite{P}, we know that $\text{dim}_H\left(\sigma_{\lambda,\alpha}\right)$ is analytic in $\lambda>0$, and from \cite{M}, we have that $\text{dim}_H\left(\sigma_{\lambda,\alpha}\right)=\text{dim}_B\left(\sigma_{\lambda,\alpha}\right)$ for all $\lambda>0$.

For large coupling constant and general frequency $\alpha$, when $\lambda>20$ it is known from \cite{LW} that $\text{dim}_H\left(\sigma_{\lambda,\alpha}\right)>0$ and there are specific necessary and sufficient conditions, regarding the continued fraction expansion of $\alpha$, for when $\text{dim}_H\left(\sigma_{\lambda,\alpha}\right)<1$. When $\lambda\geq 24$, we have a similar result from \cite{LQW}, which builds off of a series of results from \cite{FLW, LPW}, regarding the upper box counting dimension of $\sigma_{\lambda,\alpha}$.
From \cite{degt}, we know $\text{dim}_H\left(\sigma_{\lambda, \frac{\sqrt{5}-1}{2}}\right)\cdot \log\lambda$ converges to $\log(1+\sqrt{5})$ as $\lambda\rightarrow\infty$. A generalization of this result for Lebesgue a.e. $\alpha\in [0,1]\setminus \mathbb Q$ was given in \cite{CQ} where it was also shown that there is a set of full Lebesgue measure $\tilde {\mathbb I}\subset [0,1]\setminus \mathbb Q$ such that for each $(\alpha,\lambda)\in \tilde {\mathbb I}\times [24, \infty)$, the Hausdorff and box-counting dimensions of $\sigma_{\lambda,\alpha}$ coincide and are independent of $\alpha$.

For weak coupling constant, many dimension results about the spectrum have been obtained through the analysis of \textit{trace map formalism}. There is a fundamental link between the spectrum $\sigma_{\lambda, \alpha}$ and the dynamics of the so-called \textit{trace maps} over a family of cubic surfaces. More precisely, for each $a\in\mathbb{N}$, there is a corresponding trace map $T_{a}$ which leaves invariant a corresponding family of cubic surfaces $\{S_\lambda\}_{\lambda>0}$, and the dynamics of the trace maps over $S_{\lambda}$ in the prescribed order that comes exactly from the continued fraction expansion of $\alpha$, gives rich information about the spectrum $\sigma_{\lambda,\alpha}$ (see \cite{S87, Ra}). For example, the continued fraction expansion of $\alpha=\frac{\sqrt{5}-1}{2}$ contains only 1's, so here, only the Fibonacci trace map $T_1$ plays an important role. 
Many authors have examined dynamical properties of this trace map (see \cite{BGJ, BR, C, Ca, DG1, dg2, dg3, dg4, dg5, KKT}), where questions concerning hyperbolicity, Markov partitions, entropy, and more have been addressed. It turns out that \emb{the set of points with bounded orbit under the} Fibonacci trace map restricted to any $S_{\lambda}$, $\lambda>0$, is a locally maximal hyperbolic set that is homeomorphic to a Cantor set. This was proved for $\lambda\geq 16$ in \cite{C}, for small $\lambda$ in \cite{DG1}, and for all $\lambda$ in \cite{Ca}. When $\alpha$ is eventually periodic, one needs to investigate the non-wandering set of a hyperbolic map over $S_{\lambda}$, where this map is precisely a composition of finitely many trace maps, in which case several well-developed techniques from hyperbolic dynamics apply. 

In this paper, to work with the weakened assumption that $\alpha$ is of bounded-type, we incorporate non-stationary versions of various dynamical notions such as stable manifolds, Markov rectangles, and a bounded distortion estimate from \cite{BM}, \emb{since many of these classical results cannot be directly applied.}

\subsection{Background Material}
A main ingredient for achieving such results about the spectrum $\sigma_{\lambda, \alpha}$ is to investigate the dynamics of the so-called trace maps. \emb{Consider the maps $G, H: \mathbb R^3 \rightarrow\mathbb R^3$ via 
$$G(x_1,x_2,x_3):=(2x_1x_3-x_2 , x_1, x_3) \ \text{and} \ H(x_1,x_2,x_3):=(x_1,x_3,x_2).$$
For $a\in\mathbb N$, we define the $a^{\text{\textit{th}}}$\textit{-trace map} to be $T_a:\mathbb R^3\rightarrow \mathbb R^3$ via \begin{equation}\label{auxiliary map decomposition}
T_a=G^a\circ H.
\end{equation}}
For $\lambda>0$, we consider the family of cubic surfaces given by
$$S_\lambda:=\left \{ (x_1, x_2,x_3)\in\mathbb{R}^3 : x_1^2+x_2^2+x_3^2-2x_1x_2x_3=1+\frac{\lambda^2}{4} \right\}.$$
Each surface $S_{\lambda}$ is invariant under each trace map. Given a sequence $\overline a=(a_n)$, $a_n\in\mathbb{N}$, denote 
$$\Lambda_{\lambda}^{\text{bnd}}\left(\overline a \right):= \left\{\boldsymbol x\in S_{\lambda} : \left(T_{a_n}\circ\cdots \emb{\circ}T_{a_1}(\boldsymbol x)\right)_{n\in\mathbb{N}} \ \text{is bounded} \right\}.$$
Given irrational $\alpha\in(0, 1)$, we denote its continued fraction expansion by
$$
\alpha = \cfrac{1}{a_1 + \cfrac{1}{a_2 + \cfrac{1}{a_3 + \cdots}}} =: [a_1,a_2,a_3,\ldots]
$$
where each $a_k\in \mathbb{N}$ is uniquely determined. The main link between the trace maps and the spectrum $\sigma_{\lambda, \alpha}$ is the following:
\begin{theorem}\label{trace map bounded orbits connection}
For $\lambda>0$ and $\alpha=[a_1,a_2,a_3,\ldots]$ irrational, consider the sequence $\overline a=(a_n)_{n\in\mathbb{N}}$. 
Then, a real number $E$ belongs to $\sigma_{\lambda, \alpha}$ if and only if 
$$\left(\frac{E-\lambda}{2}, \frac{E}{2}, 1\right)\in \Lambda_{\lambda}^{\text{bnd}}\left( \overline a \right).$$
\end{theorem}
This Theorem is due to a combination of results from \cite{bist} and  \cite{D2000}. If we consider the line
$$L_{\lambda}:=\left\{ \left(\frac{E-\lambda}{2}, \frac{E}{2}, 1\right) : E\in\mathbb{R}  \right\},$$
which is contained in $S_{\lambda}$, then our goal is to estimate the Hausdorff dimension of $L_{\lambda}\cap  \Lambda_{\lambda}^{\text{bnd}}\left( \overline a \right)$ as $\lambda\rightarrow 0$ (see Figure \ref{Cubi Surface Figure}).

\begin{center}
            
            \begin{figure}
                \centering
                \includegraphics[scale=.5]{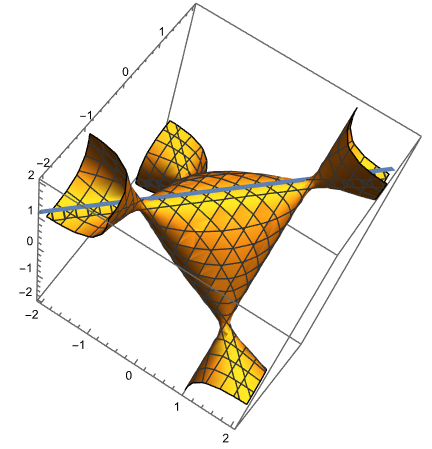}
                \caption{\emb{Surface $S_{\frac{1}{4}}$ and line $L_{\frac{1}{4}}$.}}\label{Cubi Surface Figure}
                
            \end{figure}
        \end{center}

\subsection{\emb{Approach and Outline of Paper}}\label{Outline of Paper}
\emb{Consider the points $P_1=(1,1,1), \ P_2=(-1,-1,1) , \ P_3= (1,-1,-1),$ and $P_4=(-1,1,-1)$. For $\rho>0$, let us set $O_\rho:=\bigcup_{i=1}^4 B_{\rho}\left(P_i\right),$ so that for small enough $\lambda$, we will have that $S_{\lambda}\setminus O_{\rho}$ is composed of five disjoint components, only one of which is bounded (see Figure \ref{Cubic surface with neighborhoods removed}). Denote this bounded component by $\mathbb S_{\lambda, O_{\rho}}$. For a fixed sequence $\overline a=(a_n)$, $a_n\in\mathbb N,$ and for small enough $\lambda$, we define 
$$\Lambda^{\text{surv}}_\lambda\left(\overline a, O_{\rho}\right):=\left\{\boldsymbol x\in \mathbb S_{\lambda, O_{\rho}}: T_{\lambda, a_n}\circ\cdots\circ T_{\lambda, a_1}(\boldsymbol x)\in \mathbb S_{\lambda, O_{\rho}}, \ \text{for all} \ n\in\mathbb N \right\}.$$
This set can be thought of as a non-stationary survival set. We will then prove
}
\emb{\begin{theorem}\label{first intermediate theorem}
Let $\alpha=[a_1,a_2,\dots]\in(0,1)$ be an irrational of bounded type and set $\overline a=(a_n)$. For each $\epsilon>0$, there is a $\rho>0$ and $\lambda_0>0$ such that 
$$\text{dim}_H\left(L_{\lambda}\cap \Lambda^{\text{surv}}_\lambda\left(\overline a, O_{\rho}\right) \right)\geq 1-\epsilon$$
for all $\lambda\in[0,\lambda_0]$.
\end{theorem}}

\emb{Since $\Lambda^{\text{surv}}_{\lambda}\left(\overline a, O_{\rho}\right)\subset \Lambda_{\lambda}^{\text{bnd}}\left(\overline a \right)$, together with Theorem \ref{trace map bounded orbits connection}, this will immediately imply Theorem \ref{main theorem}.}

\begin{center}
            
            \begin{figure}
                \centering
                \includegraphics[scale=.5]{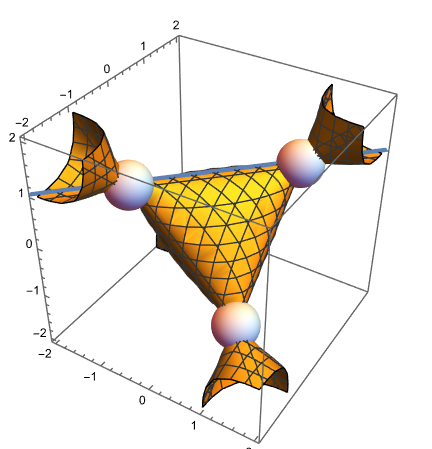}
                \caption{\emb{Depiction of $S_{\lambda}\setminus O_{\rho}$.}}\label{Cubic surface with neighborhoods removed}
                
            \end{figure}
        \end{center}
\emb{The motivation for studying the set $\Lambda^{\text{surv}}_{\lambda}\left(\overline a, O_{\rho}\right)$ comes from the fact that near each $P_i$, the curvature of the surface $S_{\lambda}$ becomes large} as $\lambda\rightarrow 0$. \emb{This makes analyzing the trace map dynamics near these cusps complicated.}

For any \emb{point} that has a bounded orbit, its orbit may enter a neighborhood of one of the cusps of $S_{\lambda}$ for some amount of time but then return to a central portion of the surface where the dynamics are comparable to those on $S_0$. This behavior results in the need for technical distortion estimates, and makes obtaining a common cone condition for our family of trace maps, over these cusps, rather difficult. These details have been written in Section 3 of \cite{dg3} for the map $T_1$ only.

\emb{So, to} avoid this obstacle, we will first remove neighborhoods of $S_\lambda$ near the cusps, \emb{and instead study the set of points whose orbits remain (or survive) in the left over bounded component}. The points from $L_{\lambda}$ that satisfy this behavior form a subset of $L_{\lambda}\cap  \Lambda_{\lambda}^{\text{bnd}}\left( \overline a \right)$, and it is sufficient to show that this subset has Hausdorff dimension near $1$ as $\lambda$ approaches $0$. Provided $\lambda$ is sufficiently small, it turns \emb{out} that each trace map restricted to a desired central portion of $S_{\lambda}$ is a factor of an Anosov map restricted to \emb{the two torus $\mathbb{T}^2:=\mathbb R^2/\mathbb Z^2$} with a small open set removed, and this map is a perturbation of a hyperbolic linear automorphism. \emb{Thus, as we will see later in the paper (see Section \ref{Main Estimates and Non-Stationary Trace Map Dynamics}), studying the set $\Lambda^{\text{surv}}_{\lambda}\left(\overline a, O_{\rho}\right)$  can be translated to a problem involving a sequence of toral Anosov maps.}

The structure of the paper is outlined as follows:
\begin{itemize}
    \item Section \ref{developing trace maps over torus} is dedicated to analyzing basic properties of the trace maps $T_{a}$ restricted to each surface $S_{\lambda}$ for a small range of $\lambda$, and we also develop some tools and notations that will be used later in the main proof. 
    \item \emb{In Section \ref{Main Estimates and Non-Stationary Trace Map Dynamics}, we state and prove Theorem \ref{intermediate theorem}, which is a translation of Theorem \ref{first intermediate theorem} to a problem concerning the non-stationary dynamics of a sequence of toral maps. The proof of this statement is done through the entirety of Section \ref{Main Estimates and Non-Stationary Trace Map Dynamics}, and hence, we give an outline of the argument immediately after the statement of Theorem \ref{intermediate theorem}.}
\item \emb{In Section \ref{Concluding the Proofs}, we use Theorem \ref{intermediate theorem} to conclude the proofs of Theorem \ref{first intermediate theorem} and Theorem \ref{main theorem}.}
\end{itemize}
As a choice of notation, we will always use "$\text{dist}$" to denote \emb{the minimal distance between sets and "$\text{dist}_H$" to denote} the Hausdorff distance between sets and $d_{C^0}$ to denote the $C^0$ distance between mappings. With a slight abuse of notation, we will denote $(x,y)$ to be the standard coordinates in both $\mathbb{T}^2$ and $\mathbb R^2$.

\section{Trace Map Dynamics as Anosov Diffeomorphisms of $\mathbb{T}^2$}\label{developing trace maps over torus}
We now relate the dynamics of our trace maps restricted to $S_{\lambda}$ to the dynamics of a family of mappings of $\mathbb{T}^2$ \emb{and derive several properties concerning these toral maps.} We will denote $\pi: \mathbb{R}^2\rightarrow \mathbb{T}^2$ to be the universal covering map given by $(x,y)\mapsto (x,y) \ (\mathrm{mod} \ 1)$.
\subsection{Torus Model for $\lambda=0$}\label{torus model for zero coupeling constant}
\emb{Notice that the} surface $S_0$ is smooth everywhere except at \emb{each $P_i$, $i=1,\dots,4$}. Let us denote $\mathbb{S}:=S_0\cap \{|x_1|\leq 1, |x_2|\leq 1, |x_3|\leq 1\}$. We know
\begin{lemma}\label{semiconjugacy lemma}
    $T_a\restriction_{\mathbb S}$ is a factor of the linear toral hyperbolic map $\hat T_{0,a}:\mathbb{T}^2\rightarrow \mathbb{T}^2$ which is induced by the matrix $A_a=\begin{pmatrix}
        a & 1 \\
        1 & 0
    \end{pmatrix}$, via the semiconjugacy $F(x, y)=(\cos(2\pi(x+y), \cos(2\pi x), \cos(2\pi y))$.
\end{lemma}
\begin{proof}
    This is a direct calculation \emb{which can be found in Lemma 3.15 in \cite{M}.}
\end{proof}

For each $a\in\mathbb{N}$, $A_a$ has two perpendicular eigenspaces with corresponding real eigenvalues for which one is less than $1$ in absolute value, and the other is greater than $1$. We call these respective eigenspaces the \textit{stable} and \textit{unstable} eigenspace of $A_a$ and denote them by $E^s_a$ and $E^u_{a}=\left(E^s_{a}\right)^{\perp}$. The corresponding \textit{stable} and \textit{unstable} eigenvalues are given by 
$$\mu^s(a):= \frac{a-\sqrt{\emb{4+a^2}}}{2}\ \text{and} \ \mu^u(a):=\frac{a+\sqrt{\emb{4+a^2}}}{2}.$$

For each $a>1$, the unstable eigenspace $E^u_a$ lies between the lines $E_1^u$ and the line $\{y=0\}$ \emb{in the projective sense}.

Let $\beta\in(0\emb{,}0.1]$, and consider the vector spaces $V_1(\beta):=\text{span}\left \{ \begin{pmatrix}
    1 \\
    -\beta
\end{pmatrix} \right\}$ and $V_2(\beta):=\text{span}\left \{ \begin{pmatrix}
    1 \\
    1+2\beta
\end{pmatrix} \right\}$. We let $K^u_{\beta}$ be the collection of vectors that are between $V_1(\beta)$ and $V_2(\beta)$, which indeed  contains each $E^u_{a}$. Formally, 
$$K^u_{\beta}:=\left\{c_1 \begin{pmatrix}
    1 \\
    -\beta
\end{pmatrix}  +c_2 \begin{pmatrix}
    1 \\
    1+2\beta
\end{pmatrix}:  c_1c_2\geq0 \right\}$$
and we also set $V_0(\beta)$ to be the line mid-way between $V_1(\beta)$ and $V_2(\beta)$. \emb{That is, $$V_0(\beta):=\text{span}\left\{ \begin{pmatrix}
    1 \\
    -\beta
\end{pmatrix}  + \begin{pmatrix}
    1 \\
    1+2\beta
\end{pmatrix}\right\}.$$}
Then, we have the following:
\begin{lemma}\label{linear cone condition}
    For each $\beta\in (0, 0.1]$, for the cones given by $K^u_{\beta}$ and $K^s_{\beta}:=\left(K_{\beta}^u\right)^{\perp}$, there exists a $\overline{\mu}_{\beta}>1$ such that for each $a\in \mathbb{N}$, we have
    \begin{itemize}
        \item $A_a^{-1}\left( K^s_\beta \right)\subset \text{int}\emb{\left(K^s_\beta\right)}$ and $A_{a}\left(K^u_\beta\right)\subset \text{int}\emb{\left(K^u_\beta\right)}$
        \item $\overline{\mu}_{\beta}|v|\leq |A_a v|\leq \mu^u(a)|v|$ for all $v\in K_{\beta}^u$
        \item $\overline{\mu}_{\beta}|v|\leq \left |A_a^{-1} v  \right|\leq \left(\emb{\left|\mu^s(a)\right|}\right)^{-1}|v|$ for all $v\in K_{\beta}^s$
        
    \end{itemize}
\end{lemma}

\begin{proof}
    Let $a\in\mathbb{N}$. To see that $A_{a}\left(K^u_\beta\right)\subset \text{int}K^u_\beta$, it suffices to check that $A_a v\in \text{int}K^u_\beta$ for the vectors $v\in\partial K^u_{\beta}$. The lines that make up $\partial K^u_\beta$ are $V_1(\beta)=\text{span}\left \{ \begin{pmatrix}
    1 \\
    -\beta
\end{pmatrix} \right\}$ and $V_2(\beta)=\text{span}\left \{ \begin{pmatrix}
    1 \\
    1+2\beta
\end{pmatrix} \right\}$. A simple calculation gives that 
$$A_a \begin{pmatrix}
            1 \\
            -\beta
        \end{pmatrix}=\begin{pmatrix}
            a-\beta \\
            1
        \end{pmatrix}
        =\left(a-\beta\right)\begin{pmatrix}
            1 \\
           \frac{1}{a-\beta}
        \end{pmatrix}.$$
    
Clearly $\frac{1}{a-\beta}>-\beta$ and also, $\beta\in(0,0.1]$ implies 
$$(1+2\beta)\left(a-\beta\right)> 1 \Rightarrow 1+2\beta> \frac{1}{a-\beta}$$
and hence $\left(a-\beta\right)\begin{pmatrix}
            1 \\
           \frac{1}{a-\beta}
        \end{pmatrix}$ belongs to a line strictly between $V_1(\beta)=\text{span}\left \{ \begin{pmatrix}
    1 \\
    -\beta
\end{pmatrix} \right\}$ and $V_2(\beta)=\text{span}\left \{ \begin{pmatrix}
    1 \\
    1+2\beta
\end{pmatrix} \right\}$ \emb{in the projective sense}. A similar calculation shows that the same is true for $ A_a \begin{pmatrix}
            1 \\
            1+2\beta
        \end{pmatrix}$.
The weakest expansion of any vector from $K^u_\beta$ under any $A_a$ is given by iterating vectors furthest away from $E_1^u$ under the map $A_1$. In other words, for any $a$ and $v\in K^u_{\beta}$,
\emb{$$\left|A_a\frac{v}{|v|} \right|\geq \left|A_1\frac{v}{|v|} \right|\geq \left|A_1\frac{\left(1,-\beta\right)^T}{\left| \left(1,-\beta\right)^T \right|} \right|=\frac{\sqrt{\left(1-\beta\right)^2+1}}{\sqrt{\beta^2+1}}=:\overline{\mu}_{\beta}.$$}
The vectors from $K^u_\beta$ that are expanded the most under $A_a$ are the ones from $E^u_a$, and these vectors are expanded by a factor of precisely $\mu^u(a)$. It follows that
$$  \overline{\mu}_{\beta}|v|\leq \left |A_a v \right| \leq \mu^u(a)|v|.$$
The invariance of $K^s_{\beta}$ under $A_a^{-1}$ and expansion estimates are checked similarly.
\end{proof}
\emb{\begin{remark}
    The range of $\beta$ in this lemma can most certainly be increased, but for our arguments, we just need any interval containing $0$.
\end{remark}}

Let $\theta_0(\beta)$ denote the angle between boundary lines of $K^u_\beta$.
Notice that $\theta_0(\beta)\rightarrow \frac{\pi}{4}$, and $\overline{\mu}_\beta\rightarrow \sqrt{2}$, as $\beta\rightarrow 0$.
For the rest of the paper, we will fix a $\beta_0\in(0, 0.1)$, so that 
\begin{equation}\label{initial beta requirements}
    \theta_0(\beta_0)<\frac{\pi}{3} \ \text{and} \ \overline{\mu}_{\beta_0}\geq \sqrt{2}-0.01,
\end{equation}
and denote 
\begin{equation}\label{main cone and line}
    K^u:=K^u_{\beta_0}, \ K^s:=K^s_{\beta_0}, \ \text{and} \ V:=V_0(\beta_0).
\end{equation}

\begin{lemma}\label{slope of linear stable manifolds}
   Let $\alpha=[a_1, a_2,\dots ]\in(0,1)$ be irrational. The set
    $$E^s_{\alpha}:=\bigcap_{n=1}^{\infty}A_{a_1}^{-1}\cdots A_{a_n}^{-1}\left(K^s\right)$$
    is a line through $0$ of slope $-\frac{1}{\alpha}$.
\end{lemma}

\begin{proof}
    For each $n\in\mathbb{N}$, we consider $\hat A_{a_n}$ to be the action on \emb{real projective space} $\mathbb R\mathbb P^1$ induced by $A_{a_n}$. Denote $\hat K^s\subset \mathbb R\mathbb P^1$ to be the set obtained from projecting $K^s$ onto $\mathbb R \mathbb P^1$, which is compact. Then, from Lemma \ref{linear cone condition} we have that each map $\hat A_{a_n}^{-1}$ restricted to $\hat K^s$ maps $\hat K^s$ into $\hat K^s$ and is a contraction with contraction rate no greater than the contraction rate of $\hat A_{1}$, which is $\mu^s(1)^2$. By Proposition \ref{nonstationary contractions}, there is a $L^*\in \hat K^s $ such that 
    $$\lim\limits_{n\rightarrow \infty} \hat A_{a_1}^{-1}\circ\cdots\circ \hat A_{a_n}^{-1}(L)=L^*,$$
    for any $L\in \hat K^s$. That is to say, $E^s_{\alpha}$ projects uniquely onto $L^*$ and is hence a line. In particular, we know that the projection of $\text{span}\left\{(0,1)^{T}\right\}$ onto $\mathbb R\mathbb P^1$, denoted by $\hat{(0,1)}$, lives in $\hat K^s$, and hence 
    $\lim\limits_{n\rightarrow\infty}\hat A_{a_1}^{-1}\circ \cdots \circ \hat A_{a_n}^{-1}\left( \hat{(0,1)}  \right)=L^*.$
    We however see that the line $\text{span}\left\{A_{a_1}^{-1}\circ\cdots\circ A_{a_n}^{-1}(1,0)^{T}\right\}$ projects onto $\hat A_{a_1}^{-1}\circ \cdots \circ \hat A_{a_n}^{-1}\left( \hat{(1,0)}  \right)$. 
    \begin{claim}\label{slope estimate}
       The line $\text{span}\left\{A_{a_1}^{-1}\circ\cdots\circ A_{a_n}^{-1}\begin{pmatrix}
0\\
1
\end{pmatrix}\right\}$ has slope $-\frac{1}{[a_1,a_2,\dots, a_n]}$.
    \end{claim}
    \begin{proof}[Proof of Claim \ref{slope estimate}.]
    First, notice that for any number $b$ and any $a\in\mathbb{N}$, one has
    $$A^{-1}_a\begin{pmatrix}
-b\\
1
\end{pmatrix}=\begin{pmatrix}
0 & 1\\
1 & -a
\end{pmatrix}\begin{pmatrix}
-b\\
1
\end{pmatrix}=(a+b)\begin{pmatrix}
-\frac{1}{a+b}\\
1
\end{pmatrix}.$$
An induction argument then gives that 
$$A_{a_1}^{-1}\cdots A_{a_n}^{-1}\begin{pmatrix}
-b\\
1
\end{pmatrix}=(\text{Constant})\begin{pmatrix}
-\alpha_n(b)\\
1
\end{pmatrix}$$
where 
\[
\alpha_n(b)=\frac{1\kern6em}{\displaystyle
  a_1 + \frac{1\kern5em}{\displaystyle
    a_2 +\stackunder{}{\ddots\stackunder{}{\displaystyle
      {}+ \frac{1}{\displaystyle
        a_{n-1} + b}}
}}}
.\]
It follows that 
$$\text{span}\left\{A_{a_1}^{-1}\circ\cdots\circ A_{a_n}^{-1}\begin{pmatrix}
0\\
1
\end{pmatrix}\right\}=\text{span}\left\{  \begin{pmatrix}
-\frac{1}{[a_1,\dots,a_n]}\\
1
\end{pmatrix} \right\},$$
and hence has slope $-\frac{1}{[a_1,\dots, a_n]}$.

\end{proof}
\emb{This shows that $E^s_{\alpha}$ has slope $-\frac{1}{\alpha}$, proving the lemma.}
\end{proof}
\emb{It is clear that} each line $\pi\left(E^s_{\alpha}\right)$ is dense in $\mathbb{T}^2$. If $Q$ is a rational point in $\mathbb{T}^2$, and $L\subset \pi\left(E^s_{\alpha}\right)$ is a line segment with midpoint $0$, the length of $L$ and continued fraction expansion of $\alpha$ determine the distance from $L$ and $Q$. \emb{This exact distance can be estimated using Proposition \ref{Circle rotation}. In particular, we have:}

\begin{lemma}\label{estimate for irrational lines}
Let $Q$ and $Q'$ be distinct rational points on $\mathbb{T}^2$, and suppose that $\mathcal A\subset\mathbb{N}$ is a finite set. 

\begin{itemize}
    \item[1.)] For each $C^*>0$, there is a $l^*=l^*\left(C^*,\mathcal A\right)>0$ such that if $L$ is a line with midpoint $Q$, irrational slope $[a_1,a_2,\dots]$, $a_i\in\mathcal A$, and $\text{length}(L)>l^*$, then $\text{dist}(L, Q')<C^*$

    \item[2.)] For each $l>0$, there is a constant $C^{**}=C^{**}\left(l,Q, Q',\mathcal A \right)>0$ such that if $L$ is a line with midpoint $Q$, irrational slope $[a_1,a_2,\dots]$, $a_i\in\mathcal A$, and $\text{length}(L)<l$, then $\text{dist}(L, Q')>C^{**}$
\end{itemize}

\end{lemma}

\begin{proof}
Without loss of generality, let us assume $Q=0$.

Let $C^*>0$. By Proposition \ref{Circle rotation}, there is a $n=n\left(C^*,\mathcal A\right)$ so that for any $\alpha=[a_1,a_2,\dots]$ where $a_i\in\mathcal A$, we must have 
$$\min\limits_{1\leq k\leq n} |k\alpha-y|<C^*$$
for any $y\in S^1$.

Consider the vertical line through $Q'=(x_0,y_0)$ given by $\text{vert}(Q')=\left\{(x_0, y) : y\in S^1 \right\}$. Then, there is a $l^*$ dependent on $n$ such that any line $L$ with midpoint $0$ and irrational slope $\alpha=[a_1,a_2,\dots]$, where $a_i\in\mathcal A$, satisfying $\text{length}(L)>l^*$ must intersect $\text{vert}(Q')$ at least $2n$ many times. 

At least $n$-many of the intersection points take the form $(x_0,z_0+k\alpha)$, $1\leq k\leq n$, for some $z_0\in S^1$. It follows that 
\begin{align*}
    &\text{dist}(L, Q')\leq \min\limits_{1\leq k \leq n}|(x_0,y_0)-(x_0,z_0+k\alpha)|\\
    &=\min\limits_{1\leq k \leq n}|(y_0-z_0)-k\alpha|<C^*,
\end{align*}
so that (1) holds.

Let $l>0$ and consider the set of lines $\mathbb L(\mathcal A)$ in $\mathbb{R}^2$ that have midpoint $0$, length exactly $l$, and irrational slope of the form $[a_1, a_2,\dots]$, $a_i\in\mathcal A$. If we consider the natural projection of $p:\mathbb{R}^2\setminus \{0\}\rightarrow \mathbb R\mathbb P^1$, then $p(\mathbb L(\mathcal A)\setminus \{0\})$ is a Cantor set on the circle $\mathbb R\mathbb P^1$, and is hence closed. By continuity of $p$, we must have that $p^{-1}(p(\mathbb L(\mathcal A)\setminus\{0\}))\cup \{0\}$ is a closed set in $\mathbb R^2$, and since $\mathbb L(\mathcal A)=\left[p^{-1}(p(\mathbb L(\mathcal A)\setminus\{0\}))\cup \{0\}\right]\cap B_{l}(0)$, it must be a compact set. Since the lines in $\mathbb L(\mathcal A)$ have irrational slope, this set is disjoint from the finite set $\pi^{-1}\left(Q'\right)\cap B_{2l}(0)$. Hence, there is a $C^{**}_0>0$ such that 
    $$\text{dist}\left( \mathbb L(\mathcal A), \pi^{-1}\left(Q'\right)\cap B_{2l}(0)\right)\geq C^{**}_0, $$
    and the result follows from here.

\end{proof}

\begin{remark}\label{remark for estimate or irrational lines}
    We note that by a symmetric argument, Lemma \ref{estimate for irrational lines} holds if the lines $L$ are replaced with lines normal to $L$ of equal length with midpoint $0$. In other words, the statement holds if we instead require the lines $L$ to have slope $-\frac{1}{[a_1,a_2,\dots]}$ such that $a_i\in \mathcal A$.
\end{remark}
\subsection{Torus Model for Small $\lambda>0$}\label{torus model for small lambda}
Let us set $Q_i:=F^{-1}(\{P_i\})$, for $i=1,\dots, 4$, so that $Q_1=(0,0)$, $Q_2=\left(\frac{1}{2},0\right)$, $Q_3=\left(\frac{1}{2},\frac{1}{2}\right)$, and $Q_4=\left(0,\frac{1}{2}\right)$. The next Lemma relates the dynamics of each $T_a$ on $S_{\lambda}$ to the dynamics of a toral map.
\begin{lemma}\label{connection}
    For each $\rho>0$ and $a\in\mathbb{N}$, there is a $\tilde \lambda(\rho, a)>0$ and $r(\rho)>0$, such that if $\lambda\in\left[0,\tilde\lambda(\rho, a)\right]$, then

    \begin{itemize}
    \item[(1)] For $O_\rho:=\bigcup_{i=1}^4 B_\rho(P_i)$, the set  $S_{\lambda}\setminus O_{\rho}$ consists of five disjoint connected components, only one of which is bounded, which we will denote by $\mathbb{S}_{\lambda,O_{\rho}}$. Furthermore, there is a smooth injective map $\pi_\lambda: \mathbb{S}_{\lambda,O_{\rho}}\rightarrow \mathbb S$, that only depends on $\lambda$, that approaches the identity map in the $C^1$-sense as $\lambda\rightarrow 0$

        \item[(2)] There is an open set $\tilde O_{r(\rho)}\subset \mathbb T^2$ which consists of four disjoint neighborhoods, each centered at a distinct $Q_i$ and of diameter less than $\frac{r(\rho)}{2}$, and a map $$\tilde T_{\lambda, a}:\mathbb{T}^2\setminus \tilde O_{r(\rho)}\rightarrow \mathbb{T}^2$$ satisfying
        $$F_{\lambda}\circ \tilde T_{\lambda, a}=T_{a}\restriction_{\mathbb S_{\lambda, O_\rho}}\circ F_{\lambda},$$
        where $F_{\lambda}:=\pi^{-1}_\lambda\circ F$ \emb{and $F$ is defined in Lemma \ref{semiconjugacy lemma}}.

    \end{itemize}
    Moreover, we have $\left\|\tilde T_{\lambda, a}-\hat T_{0,a}\right\|_{C^2}\rightarrow 0$, on $\mathbb T^2\setminus \tilde O_{r(\rho)}$, as $\lambda\rightarrow 0$, and $r(\rho)\rightarrow 0$ as $\rho\rightarrow 0$.
\end{lemma}

\begin{proof}
One can choose $\rho'\in(0,\rho)$ such that for $O_{\rho}=\bigcup_{i=1}^4 B_{\rho}\emb{\left(P_i\right)}$ and $O_{\rho'}=\bigcup_{i=1}^4 B_{\rho'}\emb{\left(P_i\right)},$ we must have 
    $$T_a^{-1}\left(O_{\rho'}\right)\subset O_{\rho}.$$
    There is a $\tilde\lambda(\rho,a)>0$ such that for each $\lambda\in\left[0,\tilde\lambda(\rho,a)\right]$, we have that $S_{\lambda}\setminus O_{\rho} \ \text{and} \ S_{\lambda}\setminus O_{\rho'}$ are both unions of five disjoint connected components, only one of which is bounded.

    Denoting the bounded component of $S_{\lambda}\setminus O_{\rho'}$ by $\mathbb S_{\lambda, O_{\rho'}}$, there is a smooth injective map $\pi_{\lambda}: \mathbb S_{\lambda, O_{\rho'}}\rightarrow \mathbb{S}$, that is defined as follows: 
    
    For each $\boldsymbol y\in \mathbb S$, consider the normal line $N_{\boldsymbol y}$ to $\mathbb S$ at $\boldsymbol y$ that is pointing outward. Either $N_{\boldsymbol{y}}$ intersects $\mathbb S_{\lambda, O_{\rho'}}$ uniquely, or not at all. In the former case, we set $\pi_{\lambda}^{-1}\left(\boldsymbol y\right)=N_{\boldsymbol{y}}\cap \mathbb S_{\lambda, O_{\rho'}}$, and this gives us (1).

Denote the bounded component of $ S_{\lambda}\setminus O_{\rho}$ by $\mathbb S_{\lambda, O_{\rho}}$. Then, there is a $r(\rho)>0$ and an open set $\tilde O_{r(\rho)}\subset \mathbb T^2$ which consists of neighborhoods of $Q_i$, $i=1,\dots,4$, each of diameter less than $\frac{r(\rho)}{2}$,
    such that $F^{-1}\circ\pi_\lambda\left(\mathbb S_{\lambda, O_{\rho}}\right)=\mathbb{T}^2\setminus \tilde O_{r(\rho)}$. \emb{Also, there is an open set $O'\subset \tilde O_{r(\rho)}$ such that $F^{-1}\circ\pi_\lambda\left(\mathbb S_{\lambda, O_{\rho'}}\right)=\mathbb{T}^2\setminus O'$.}
    
    \emb{Let us denote $F_{\lambda}:=\pi_{\lambda}^{-1}\circ F$ and let $P\in \mathbb{T}^2$}. Then, $F_{\lambda}(P)$ is contained in $\mathbb S_{\lambda, O_{\rho}}$ and $\left(T_{a}\circ F_{\lambda}\right)(P)\in \mathbb S_{\lambda, O_{\rho'}}$, so that
    $$\left(F^{-1}_{\lambda}\circ T_{a}\circ F_\lambda\right) (P)$$
    is a set consisting of two points $Q, Q'\in \mathbb T^2\setminus O'$. If $\tilde \lambda(\rho,a)$ is small enough, then one of the two points in $\{Q, Q'\}$ is much closer to $\hat T_{0, a}$ than the other.
    
    Let us choose $\tilde T_{\lambda, a}(P)$ to be the element from $\{Q, Q'\}$ that is closest to $\hat T_{0,a}(P)$. Then, $\tilde T_{\lambda, a}:\mathbb T^2\setminus \tilde O_{r(\rho)}\rightarrow \mathbb T^2\setminus O'$ is well-defined, and satisfies 
    $$F_{\lambda}\circ\tilde T_{\lambda, a}=T_{a}\restriction_{\mathbb S_{\lambda, O_{\rho}}} \circ F_{\lambda}.$$

    Since $\left\|\pi_{\lambda}\circ T_a\circ \pi^{-1}_\lambda\restriction_{\mathbb S_{0, O_\rho}}-T_a\restriction_{\mathbb S_{0,O_\rho}}\right\|_{C^2}\rightarrow 0$ as $\lambda\rightarrow 0$, it follows that $\left\|\tilde T_{\lambda, a}-\hat T_{0,a}\right\|_{C^2}\rightarrow 0$, on $\mathbb{T}^2\setminus \emb{\tilde O_{r(\rho)}}$, as $\lambda\rightarrow 0$.

    Lastly, since $\text{diam}\left(B_\rho(P_i)\right)\rightarrow 0$ as $\rho\rightarrow 0$ for each $i=1,\dots, 4$, we can choose $r(\rho)$ so that $r(\rho)\rightarrow 0$ as $\rho\rightarrow 0$.

\end{proof}

\section{\emb{Dynamical Tools and Setup}}
\emb{We will now setup various tools and notations that will be helpful in the main proof.}

\emb{Namely, in Section \ref{Main Estimates and Non-Stationary Trace Map Dynamics} we will invoke Lemma \ref{connection} and translate Theorem \ref{first intermediate theorem} into a problem concerning non-stationary dynamics of Anosov maps around small neighborhoods of each $Q_i$. Rather than working with arbitrary neighborhoods, we will define a notion of dynamical rectangles, that are analogous to Markov rectangles for hyperbolic maps, that will be useful in examining the dynamics  near each $Q_i$.} \emb{This section is dedicated to precisely defining these rectangles, and studying some basic geometric properties that they exhibit.}
\subsection{\emb{Boxes and Local Cones}}\label{local cones and boxes notation}
\emb{Before defining the rectangles with dynamical properties, we will first restrict ourselves to small preliminary boxes around each $Q_i$. This subsection is dedicated to choosing these boxes, deriving some needed properties for later, and also introducing some notations.}

Recalling our definition of $V$ from (\ref{main cone and line}), consider the projections $\pi_{V}: V\oplus V^{\perp}\rightarrow V$ and $\pi_{V^{\perp}}: V\oplus V^{\perp}\rightarrow V^{\perp}$, and set $\|.\|_{V}:=\left | \pi_{V}(.) \right |$ and $\|.\|_{V^{\perp}}:=\left | \pi_{V^{\perp}}(.) \right |$.
For $\boldsymbol{Q}\in\mathbb{R}^2$, consider the rectangular neighborhood given by $$\text{Box}_{r}(\boldsymbol Q):=\left\{\boldsymbol P : \|\boldsymbol P - \boldsymbol Q \|_{V}\leq r\right\}\times \left\{\boldsymbol P : \|\boldsymbol P - \boldsymbol Q \|_{V^{\perp}}\leq r\right\}.$$ 
\emb{For these choice of sets, it is important that each map $A_a$ overflows the square  $\text{Box}_{1}(0)$. We use this fact later in Theorem \ref{stable manifold theorem} and Lemma \ref{inner rectangle construction}. Formally, we have the following:}
\emb{\begin{lemma}\label{linear overflow}
    Let $L$ be a line segment through $0$ that is contained in $K^u$ and connects the boundary lines of $\text{Box}_{1}(0)$ that are parallel to $V^{\perp}$. Then, for any $a\in\mathbb N$, we have that $A_a(L)$ is a line through $0$ that is contained in $K^u$ and extends outside of the square $\text{Box}_{1}(0)$ (see Figure \ref{overflow illustration}).\\
\indent Moreover, for each $\beta>\beta_0$ sufficiently close to $\beta_0$, there is a $\epsilon=\epsilon(a,\beta)>0$ such that if $f\in \text{Diff}^1\left(\mathbb R^2\right)$ and 
$\left\|A_a-f\right\|<\epsilon,$
then $f(L)$ is a curve that overflows $\text{Box}_{1}(0)$. That is, $f(L)$ is contained in $K^u_{\beta}$, and its endpoints belong to the compliment of $\text{Box}_{1}(0)$.
\end{lemma}
\begin{proof}
    Suppose that $a\in\mathbb N$ and consider the square $\text{Box}_{1}(0)$. Then, from our choice of $\beta_0$ (see \ref{main cone and line}), we know that the length of any segment $L$ contained in $K^u$ that connects $0$ and a boundary line of $\text{Box}_{1}(0)$ parallel to $V^\perp$, has length strictly less than $\frac{2}{\sqrt{3}}$ and length no less than $1$. If we iterate this segment under $A_a$, then its length will be greater than $\overline \mu_{\beta_0}$. However, from (\ref{initial beta requirements}), we have 
$$\overline \mu_{\beta_0}\geq\sqrt{2}-0.01> \frac{2}{\sqrt{3}}.$$
It follows that $A_{a}(L)$ is a line through $0$ that is contained in $K^u$ and extends outside of the square $\text{Box}_{1}(0)$.\\
From the discussion before (\ref{initial beta requirements}), for $\beta>\beta_0$ close enough to $\beta$, we can still guarantee that the length of any segment $L$ contained in $K^u_{\beta}$ that connects $0$ and a boundary line of $\text{Box}_{1}(0)$ parallel to $V^\perp$, has length strictly less than $\frac{2}{\sqrt{3}}$. Then, by shrinking $\beta$ is needed, we have that 
$$\overline \mu_{\beta}\geq (\sqrt{2}-0.01)-0.001$$
and hence, this expansion is still greater than $\frac{2}{\sqrt{3}}$. If $L$ is such a segment and $f$ is sufficiently $C^1$ close to $A_a$, then the endpoints of $f(L)$ will be close to those of $A_a(L)$, and we can guarantee that $Df\left(K^u_{\beta}\right)\subset K^u_\beta$. This would imply that the tangent line at each point of $f(L)$ is contained in $K^u_{\beta}$, and hence $f(L)\subset K^u_{\beta}$.
\end{proof}}

\emb{We now define some sets and notations that will be used frequently throughout the paper.}

\begin{figure}[h]
    \begin{tikzpicture}

\filldraw[black] (0, 0) circle (1pt) node[below]{\scriptsize $0$};
\draw[black, very thick] (2,2) rectangle (-2,-2) node[anchor=north]{\small $\text{Box}_{1}\left(0 \right)$};

\draw[black, very thick] (2, .65*2) -- (0,0) -- (2,-.65*2) --(2, .65*2);
\draw[black, very thick] (-2, .65*-2) -- (0,0) -- (-2,.65*2) --(-2, .65*-2);

\fill[black, opacity=0.15] (2, .65*2) -- (0,0) -- (2,-.65*2) --(2, .65*2);

\fill[black, opacity=0.15] (-2, .65*-2) -- (0,0) -- (-2,.65*2) --(-2, .65*-2);

\draw[blue, very thick] (2, .25*2) -- (-2, .25*-2) node[below right]{\small $L$};

\draw[purple, very thick] (2.5, -.45*2.5) -- (-2.5, -.45*-2.5) node[below left]{\small $A_a(L)$};

\filldraw[black] (0, 0) circle (2pt) ;

\filldraw[blue] (2, .25*2) circle (2pt) ;
\filldraw[blue] (-2, .25*-2) circle (2pt) ;

\filldraw[purple] (2.5, -.45*2.5) circle (2pt) ;
\filldraw[purple] (-2.5, -.45*-2.5) circle (2pt) ;

\coordinate[label=below: \scriptsize $\emb{K^u}$] (C) at (1.65,-1.25);

\draw[<->]  (-5,1) -- (-5,2) node[right]{\scriptsize $V^{\perp}$};
\draw[<->]  (-5.5,1.5) -- (-4.5,1.5) node[right]{\scriptsize $V$};

\end{tikzpicture}
\caption{Illustration of \emb{Lemma} \ref{linear overflow}}
    \label{overflow illustration}

\end{figure}
Recalling that $\pi: \mathbb{R}^2\rightarrow \mathbb{T}^2$ is the universal covering map given by $(x,y)\mapsto (x,y) \ (\mathrm{mod} \ 1)$, then for each $\boldsymbol Q\in \mathbb{R}^2$, we have a local diffeomorphism 

\begin{equation}\label{local projection}
    \phi_{\boldsymbol Q}:=\pi\restriction_{\text{Box}_{\frac{1}{10}}(\boldsymbol Q)}: \text{Box}_{\frac{1}{10}}(\boldsymbol Q)\rightarrow \pi\left(\text{Box}_{\frac{1}{10}}(\boldsymbol{Q})\right),
\end{equation}

where $\pi\left(\text{Box}_{\frac{1}{10}}(\boldsymbol{Q})\right)$ is the rectangular neighborhood centered at $\pi(\boldsymbol Q)$, with sides parallel to the lines $\pi(V)$ and $\pi(V^{\perp})$, respectively, of lengths given by $\frac{1}{5}$. 

\emb{\begin{definition}
    For each $r\in \left(0, \frac{1}{10}\right)$, we define (see Figure \ref{box illustration})
    \begin{itemize}
        \item $\mathbb B_r(Q):=\pi\left(\text{Box}_r\left(\pi^{-1}(Q)\right)\right)$, which is the rectangular neighborhood centered at $Q$ of side lengths $2r$
        \item $\partial^s\mathbb B_r(Q)$ to be the two line segments of $\partial \mathbb B_r(Q)$ that are parallel to the line $\pi\left(V^{\perp}\right)$
        \item $\partial^u\mathbb B_r(Q)$ to be the two line segments of $\partial \mathbb B_r(Q)$ that are parallel to the line $\pi\left(V\right)$
    \end{itemize}
    In particular, let we set $\mathbb B_i:=\mathbb B _{\frac{1}{10}}(Q_i)$, $\partial ^s\mathbb B_i:=\partial^s\mathbb B _{\frac{1}{10}}(Q_i)$ and $\partial ^u\mathbb B_i:=\partial^u\mathbb B _{\frac{1}{10}}(Q_i)$.
\end{definition}}

\begin{figure}[h]
   \begin{center}
    \begin{tikzpicture}

\filldraw[black] (-3, 0) circle (1pt) node[below right]{\scriptsize $\boldsymbol Q$};
\draw[black, very thick] (-4,1) rectangle (-2,-1) node[anchor=north]{\small $\text{Box}_{r}\left(\boldsymbol Q \right)$};

\draw[red, dashed] (-3, 0) -- (-4, 0) node[midway, above]{\scriptsize $r$};

\filldraw[black] (3,0) circle (1pt) node[below right]{\scriptsize $Q$};
\draw[black, very thick] (2,1) rectangle (4, -1) node[anchor=north]{$\mathbb{B}_r(Q)$};

\draw[->]  (-1.7,0) -- (1.7,0) node[midway, above]{$\phi_{\boldsymbol Q}$};

\filldraw[black] (-5, 1.5) circle (1pt);

\draw[<->]  (-5,1) -- (-5,2) node[right]{\scriptsize $V^{\perp}$};
\draw[<->]  (-5.5,1.5) -- (-4.5,1.5) node[right]{\scriptsize $V$};

\draw[red, dashed] (2, 0) -- (3, 0) node[midway, above]{\scriptsize $r$};

\coordinate[label=right: \scriptsize $\partial^s\mathbb B_r(Q)$] (A) at (4,0);
\coordinate[label=above: \scriptsize $\partial^u\mathbb B_r(Q)$] (B) at (3,1);
\end{tikzpicture}

    \caption{Illustration of local boxes.}
    \label{box illustration}

    \end{center}
\end{figure}
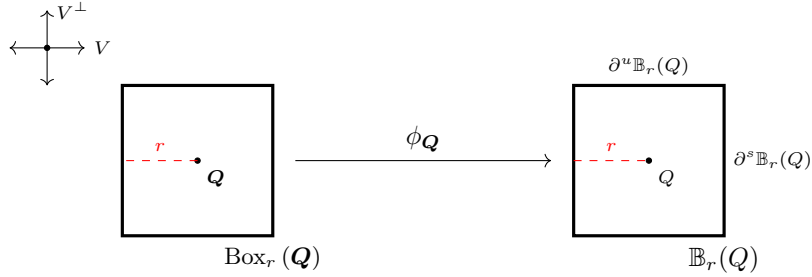

\emb{\begin{definition} (see Figure \ref{local cones}) For any cone $K$ containing $K^u$ and transversal to $K^s$, we define 
    \begin{equation}\label{+ and - cones}
    K_{-}:=K^u\cap\{x\leq 0\} \ \text{and} \ K_{+}:=K\cap \{x \geq 0\}.
\end{equation}
For a cone $K\subset\mathbb R^2$, define the local projected cone $\tilde K(Q)\subset \mathbb B_{\frac{1}{10}}(Q)$ via
$$\tilde K(Q):=\pi\left(\left(\boldsymbol Q+K\right)\cap\text{Box}_{\frac{1}{10}}(\boldsymbol Q)\right)$$
If $K$ is a cone containing $K^u$ and transversal to $K^s$, then set $$\tilde K_+(Q):=\pi\left(\left(\boldsymbol Q+K_+\right)\cap\text{Box}_{\frac{1}{10}}(\boldsymbol Q)\right) \ \text{and} \ \tilde K_-(Q):=\pi\left(\left(\boldsymbol Q+K_-\right)\cap\text{Box}_{\frac{1}{10}}(\boldsymbol Q)\right)$$
With this in mind, we define
\begin{itemize}
    \item $\partial^s_{+}\mathbb B_r(Q)$ to be the line segment from $\partial^s\mathbb B_r(Q)$ that intersects $\tilde K_{+}^u(Q)$
    \item $\partial^s_{-}\mathbb B_r(Q)$ to be the line segment from $\partial^s\mathbb B_r(Q)$ that intersects $\tilde K_{-}^u(Q)$
    \item $K^{s,u}(x,y)$ to be the cone $K^{s,u}$ contained in the tangent space $T_{(x,y)}\mathbb{T}^2$
\end{itemize}
 Lastly, we set
\begin{equation}\label{cone fields}
    \mathcal K^s:=\left\{K^s(x,y)\right\}_{(x,y)\in\mathbb{T}^2} \ \text{and} \ \mathcal K^u:=\left\{K^u(x,y)\right\}_{(x,y)\in\mathbb{T}^2},
\end{equation} and refer to these collections as \textit{stable}, respectively \textit{unstable}, \textit{cone fields}. 
\end{definition}}

\begin{figure}[h]
   \begin{center}
    \begin{tikzpicture}

\filldraw[black] (-3, 0) circle (1pt) node[below]{\scriptsize $\boldsymbol Q$};
\draw[black, very thick] (-4,1) rectangle (-2,-1) node[anchor=north]{\small $\text{Box}_{r}\left(\boldsymbol Q \right)$};

\draw[black, very thick] (4, .65*1)-- (2,.65*-1);
\draw[black, very thick] (4, -.65*1)-- (2,.65*1);

\fill[black, opacity=0.15] (4, .65*1) -- (3,0) -- (4, -.65*1)--(4, .65*1);

\fill[black, opacity=0.15] (2,-.65*1) -- (3,0) -- (2,.65*1) --(2,-.65*1);

\filldraw[black] (3,0) circle (1pt) node[below]{\scriptsize $Q$};
\draw[black, very thick] (2,1) rectangle (4, -1) node[anchor=north]{$\mathbb{B}_r(Q)$};

\draw[black, very thick] (-4, .65*-1)-- (-2,.65*1);
\draw[black, very thick] (-4, .65*1)-- (-2,-.65*1);

\fill[black, opacity=0.15] (-4, .65*-1) -- (-3,0) -- (-4, .65*1)--(-4, .65*-1);

\fill[black, opacity=0.15] (-2,.65*1) -- (-3,0) -- (-2,-.65*1)--(-2,.65*1);

\draw[->]  (-1.7,0) -- (1.7,0) node[midway, above]{$\phi_{\boldsymbol Q}$};

\filldraw[black] (-5, 1.5) circle (1pt);

\draw[<->]  (-5,1) -- (-5,2) node[right]{\scriptsize $V^{\perp}$};
\draw[<->]  (-5.5,1.5) -- (-4.5,1.5) node[right]{\scriptsize $V$};

\coordinate[label=left: \scriptsize $\boldsymbol -$] (C) at (-3.5,0);

\coordinate[label=right: \scriptsize $\boldsymbol +$] (C) at (-2.5,0);

\coordinate[label=left: \scriptsize $\boldsymbol -$] (C) at (2.5,0);

\coordinate[label=right: \scriptsize $\boldsymbol +$] (C) at (3.5,0);

\coordinate[label=below: \scriptsize $\boldsymbol Q+K$] (C) at (-3.25,1);

\coordinate[label=below: \scriptsize $\tilde K(Q)$] (C) at (3.25,1);
\end{tikzpicture}

    \caption{Illustration of a local cone.}
    \label{local cones}

    \end{center}
\end{figure}
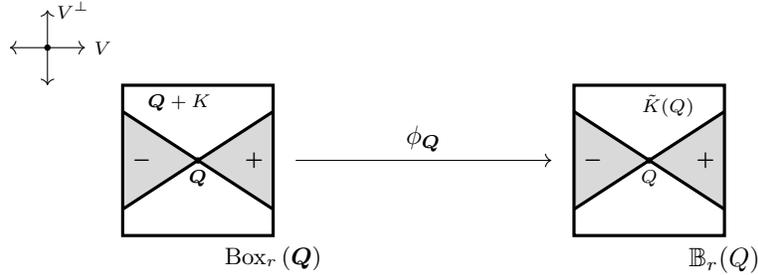

\emb{The next lemma is another property concerning these boxes that will be frequently used later in Section \ref{Main Estimates and Non-Stationary Trace Map Dynamics}.}
\emb{\begin{lemma}\label{curve length compared to distance}
    There is a $C_1>0$ such that if $\gamma\subset \mathbb T^2$ is a curve tangent to $\mathcal K^u$, and is contained in $ \mathbb B_{\frac{1}{10}}(Q)$, for some $Q\in\mathbb T^2$, with endpoints $P_1$ and $P_2$, then 
    $$\text{length}(\gamma)\leq C_1\left|P_1-P_2\right|.$$
\end{lemma}
\begin{proof}
    This follows from the fact that any curve tangent to $\mathcal K^u$ contained in $\mathbb B_{\frac{1}{10}}(Q)$ is the graph of a Lipschitz map from $\pi(V)\cap \mathbb B_{\frac{1}{10}}(Q)$ to $\pi\left(V^{\perp}\right)\cap \mathbb B_{\frac{1}{10}}(Q)$ with Lipschitz constant determined by the cone $K^u$.
\end{proof}}

\subsection{\emb{Dynamical Rectangles}}\label{rectangle section}

\emb{We are now ready to define these so-called dynamical rectangles.}

\begin{definition}
    A \textit{dynamical rectangle} (or \textit{rectangle}) is a curvilinear rectangle $R$, such that two opposite boundary curves of $R$ are $C^1$ curves tangent to $\mathcal K^s$, referred to as the \textit{stable boundaries}, denoted by $\partial^s R$, and the other two boundary curves are transversal to $\mathcal K^s$, referred to as the \textit{non-stable boundaries}, and denoted by $\partial^{ns}R$.
\end{definition}

\emb{\begin{definition}
    Let $\gamma$ be a $C^1$ curve contained in $\mathbb B_{\frac{1}{10}}(Q)$ that is tangent to $\mathcal K^s$. We say that $\gamma$
    \begin{itemize}
        \item \textit{connects opposite boundaries of} $\tilde K^u(Q)$ if its endpoints belong to $\partial \tilde K_{+}^u(Q)$
        \item \textit{connects points} $P, P'\in  \mathbb B_{\frac{1}{10}}(Q)$, if these are the endpoints of $\gamma$, and $\gamma$ is strictly contained in $\overline{  B_{\frac{1}{10}}(Q) }$
        \item \textit{connects disjoint curves} $\gamma_1, \gamma_2\subset B_{\frac{1}{10}}(Q)$ if it connects points of these curves
    \end{itemize}
\end{definition}}

\emb{\begin{definition}
    Let $Q\in\mathbb T^2$ and $K\supset K^u$ be a cone transversal to $K^s$. We say 
    \begin{itemize}
        \item Two $C^1$ curves $\gamma_1, \gamma_2\subset \mathbb B_{\frac{1}{10}}(Q)$ tangent to $\mathcal K^s$, are \textit{centered at} $Q$ \textit{with respect to} $\tilde K(Q)$, if 
           \begin{itemize}
            \item There is a sub-curve of $ \gamma_1$ that connects opposite boundaries of $\tilde K_{+}(Q)$ and does not intersect $\tilde K_{-}(Q)$
            \item There is a sub-curve of $ \gamma_2$ that connects opposite boundaries of $\tilde K_{-}(Q)$ and does not intersect $\tilde K_{+}(Q)$
        \end{itemize}
        or vice versa (see Figure \ref{dynamical rectangle}).
        \item A rectangle $R\subset B_{\frac{1}{10}}(Q)$ is \textit{centered at} $Q$ with respect to $\tilde K(Q)$ and its stable boundary curves are centered at $Q$ with respect to $\tilde K(Q) $. In this case we denote $\partial^s_+R$ to be the curve from $\partial^s R$ that has a segment connecting opposite boundaries of $\tilde K^u_+(Q)$ and $\partial^s_{-}R$ to be the other curve which has a segment that connects opposite boundaries of $\tilde K^u_-(Q)$.
        (see Figure \ref{dynamical rectangle})
    \end{itemize}
\end{definition}}

\emb{\begin{remark}
    Notice that any pair of curves centered at $Q$ with respect to $\tilde K(Q)$, must be centered at $Q$ with respect to $\tilde K^u(Q)$. If not specified otherwise, we say that two curves are centered at $Q$, if they are centered at $Q$ with respect to $\tilde K^u(Q)$.\\
    \\
Also, we denote $\partial^s_+R$ to be the curve from $\partial^s R$ that has a segment connecting opposite boundaries of $\tilde K^u_+(Q)$ and $\partial^s_{-}R$ to be the other curve which has a segment that connects opposite boundaries of $\tilde K^u_-(Q)$. It should also be noted that $\partial^s_{\pm}R\cap \tilde K^u_{\mp}(Q_i)=\emptyset$.
\end{remark}}

\subsection{\emb{Geometric Properties of Dynamical Rectangles}}
\emb{We now derive two propositions that describe useful geometric properties concerning these local cones and dynamical rectangles. The next proposition is illustrated by Figure \ref{inner constant for box figure}}.
\emb{\begin{proposition}\label{inner constant for box}
    There is a $C^*_0>0$ such that for any $i=1,\dots, 4$, the following hold:
    \begin{itemize}
        \item[(1)] If $\gamma$ is a $C^1$ curve tangent to $\mathcal K^s$ that is contained in $\overline{\mathbb B}_i$ satisfying $$\text{dist}(\gamma, Q_i)<C^{*}_0,$$ then $\gamma$ is disjoint from $\partial^s\mathbb B_i$.
        \item[(2)] If $\kappa$ is a curve tangent to $\mathcal K^u$ with endpoints $P_1$ and $P_2$ satisfying $P_1\in \tilde K_{+}(Q_i)$, $P_2\in\tilde K_-(Q_i)$, and $$\max\left\{\left|P_1-Q_i\right|, \left|P_2-Q_i\right|\right\}<C^{*}_0,$$ we must have $\kappa\subset \mathbb B_i$. That is $\kappa$ is disjoint from $\partial ^u\mathbb B_i$.
    \end{itemize}
 \end{proposition}
 }

 \begin{proof}
    \emb{We prove part (1). Since $\text{dist}\left(\partial^s\mathbb B_i, \tilde K^s(Q_i)\right)>0$, by construction of $K^s$, if $L\subset\overline {\mathbb B}_i$ is a line segment parallel to a segment of $\partial \tilde K^s(Q_i)$, and sufficiently close to $Q_i$, we must also have $\text{dist}\left(\partial^s\mathbb B_i, L\right)>0.$ It follows that the same is true for any curve $\gamma\subset \overline {\mathbb B}_i$ that is tangent to $\mathcal K^s$. The proof of (2) is similar.}
 \end{proof}
 
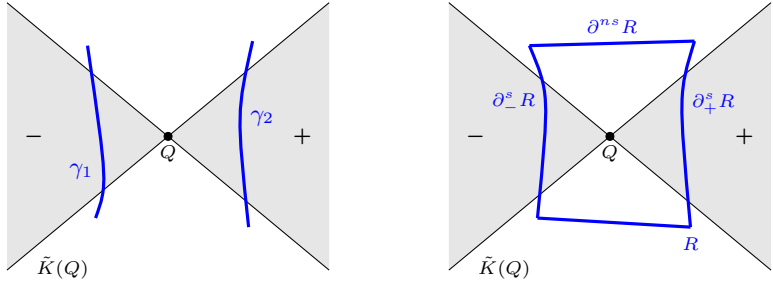
\begin{figure}
    \begin{center}
            \begin{tikzpicture}[scale=.75]
\begin{axis}[xtick=\empty, ytick=\empty, axis line style={draw=none}, tick style={draw=none}]
    \addplot[name path=B, domain={-4:4}]{0.7*x};   
\addplot[name path=A, domain={-4:4}]{-0.7*x};

\addplot[fill=black, fill opacity=0.1] fill between [of=A and B];

\draw[blue, very thick] (-1.8,0.95*-1.8)  .. controls (-1.5, -1) ..  (-2,-0.95*-2) node[midway, left]{\small $\gamma_1$};
\draw[blue, very thick] (2.1, 0.95*2.1).. controls (1.7, .5) .. (2,-0.95*2)  node[midway, right]{\small $\gamma_2$};    

\filldraw[black] (0,0) circle (2pt) node[anchor=north]{\scriptsize $Q$};

\coordinate[label=below right: \scriptsize $\tilde K(Q)$] (A) at (-3.5,-2.9*0.8);
\coordinate[label=right: \scriptsize $\boldsymbol -$] (B) at (-3.8,0);
\coordinate[label=left: \scriptsize $\boldsymbol +$] (C) at (3.8,0);
\end{axis}
\end{tikzpicture}
\hspace{.5cm}
        \begin{tikzpicture}[scale=.75]
\begin{axis}[xtick=\empty, ytick=\empty, axis line style={draw=none}, tick style={draw=none}]
    \addplot[name path=B, domain={-4:4}]{0.7*x};   
\addplot[name path=A, domain={-4:4}]{-0.7*x};

\addplot[fill=black, fill opacity=0.1] fill between [of=A and B];

\draw[blue, very thick] (-1.8,0.95*-1.8)  .. controls (-1.5, .9) ..  (-2,-0.95*-2) node[midway, left]{\scriptsize $\partial^s_-R$};
\draw[blue, very thick] (2.1, 0.95*2.1) -- (-2,-0.95*-2) node[midway, above]{\scriptsize $\partial^{ns}R$};
\draw[blue, very thick] (-1.8,0.95*-1.8) -- (2,-0.95*2) node[anchor=north]{\scriptsize $R$};
\draw[blue, very thick] (2.1, 0.95*2.1).. controls (1.7, .9) .. (2,-0.95*2)  node[midway, right]{\scriptsize $\partial^s_+R$};    

\filldraw[black] (0,0) circle (2pt) node[anchor=north]{\scriptsize $Q$};

\coordinate[label=below right: \scriptsize $\tilde K(Q)$] (A) at (-3.5,-2.9*0.8);
\coordinate[label=right: \scriptsize $\boldsymbol -$] (B) at (-3.8,0);
\coordinate[label=left: \scriptsize $\boldsymbol +$] (C) at (3.8,0);
\end{axis}
\end{tikzpicture}
\end{center}

    \caption{The \emb{left} figure depicts two curves $\gamma_1$ and $\gamma_2$ that are centered at $Q$ with respect to a cone $\tilde K(Q)$, and the \emb{right} figure depicts a dynamical rectangle $R$ centered at $Q$ with respect to $\tilde K(Q)$.}
    \label{dynamical rectangle}
\end{figure}

\begin{figure}
    \begin{center}
        \begin{tikzpicture}[scale=1.2]
\begin{axis}[xtick=\empty, ytick=\empty, axis line style={draw=none}, tick style={draw=none}]
\draw[black, very thick] (-3.5,-3.5) -- (3.5,-3.5) -- (3.5,3.5) -- (-3.5,3.5) --  (-3.5,-3.5) node[below]{\scriptsize $\emb{\overline {\mathbb B}_i}$};
\addplot[domain={-3.5:3.5}]{3.5};  
\addplot[domain={-3.5:3.5}]{-3.5};

\addplot[name path=B, domain={-3.5*0.7:3.5*0.7}]{(10/7)*x};   
\addplot[name path=A, domain={-3.5*0.7:3.5*0.7}]{-(10/7)*x}; 

\fill[black, opacity=0.2] (-3.5*0.7, 3.5) -- (3.5*0.7, 3.5) -- (0,0) -- (-3.5*0.7, 3.5);

\fill[black, opacity=0.2] (-3.5*0.7, -3.5) -- (3.5*0.7, -3.5) -- (0,0) -- (-3.5*0.7, -3.5);

\draw[red, dashed] (0.9,0) -- ({3.5*0.7+0.9},3.5) ;
\draw[red, dashed] (0.9,0) -- ({3.5*0.7+0.9},-3.5) ;
\draw[red, dashed] (-0.9,0) -- ({-3.5*0.7-0.9},-3.5) ;
\draw[red, dashed] (-0.9,0) -- ({-3.5*0.7-0.9},3.5) ;

\filldraw[red] (0.9,0) circle (2pt);
\filldraw[red] ({3.5*0.7+0.9},3.5) circle (2pt);
\filldraw[red] ({3.5*0.7+0.9},-3.5) circle (2pt);
\filldraw[red] (-0.9,0) circle (2pt);
\filldraw[red] ({-3.5*0.7-0.9},3.5) circle (2pt);
\filldraw[red] ({-3.5*0.7-0.9},-3.5) circle (2pt);
\filldraw[red] (0,1.286) circle (2pt);
\filldraw[red] (0,-1.286) circle (2pt);
\filldraw[black] (0,0) circle (2pt) node[anchor=west]{\tiny $Q$};

\draw[red, very thick] (-0.9,0) -- (0,1.286) -- (0.9,0) -- (0,-1.286) --  (-0.9,0);
 \fill[red, opacity=0.15] (-0.9,0) -- (0,1.286) -- (0.9,0) -- (0,-1.286) --  (-0.9,0);

 \coordinate[label=below: \tiny $\tilde K^s(Q)$] (A) at (0,3.4);

 \draw[blue, very thick] (-1.5, 3.5) .. controls (-.25, 0) .. (-1.5,-3.5);

 \filldraw[blue] (-1.5, 3.5) circle (2pt);
  \filldraw[blue] (-1.5, -3.5) circle (2pt);

  \coordinate[label= right: \tiny $\partial^s \mathbb B_{\frac{1}{10}}(Q)$] (A) at (-3.5,0);
  \coordinate[label= right: \small \textcolor{blue}{\small $\gamma$}] (B) at (-1.2,-3.1);
\end{axis}

\end{tikzpicture}
\hspace{.5cm}
\begin{tikzpicture}[scale=1.2]
\begin{axis}[xtick=\empty, ytick=\empty, axis line style={draw=none}, tick style={draw=none}]
    \draw[black, very thick] (-3.5,-3.5) -- (3.5,-3.5) -- (3.5,3.5) -- (-3.5,3.5) --  (-3.5,-3.5) node[below]{\scriptsize $\emb{\overline {\mathbb B}_i}$};
 \addplot[name path=C, domain={-3.5:3.5}]{0.95*x};   
\addplot[name path=D, domain={-3.5:3.5}]{-0.95*x};  

 \addplot[name path=B, domain={-3.5:3.5}]{0.7*x};   
\addplot[name path=A, domain={-3.5:3.5}]{-0.7*x};

\addplot[fill=black, fill opacity=0.2] fill between [of=A and B];
 
 \addplot[fill=black, fill opacity=0.1] fill between [of=C and D];
 \coordinate[label=below right: \tiny $\tilde K(Q)$] (A) at (-2.8,-2.8*0.95);
 \coordinate[label=right : \tiny $\tilde K^u(Q)$] (B) at (-3.5, 0);

 \filldraw[black] (0,0) circle (2pt) node[anchor=north]{\scriptsize $Q$};

 \draw[red, very thick] (0,0) -- (1.697, -1.612) -- (2, -1.4) --  (2, 1.4) -- (1.697, 1.612) -- (0,0);
\fill[red, opacity=0.15] (0,0) -- (1.697, -1.612) -- (2, -1.4) --  (2, 1.4) -- (1.697, 1.612) -- (0,0);

 \draw[red, very thick] (0,0) -- (-1.697, -1.612) -- (-2, -1.4) --  (-2, 1.4) -- (-1.697, 1.612) -- (0,0);
 \fill[red, opacity=0.15] (0,0) -- (-1.697, -1.612) -- (-2, -1.4) --  (-2, 1.4) -- (-1.697, 1.612) -- (0,0);

 \draw[red, dashed] (1.697, -1.612) -- (0,-2.8) -- (-1.697, -1.612) ;

 \draw[red, dashed] (1.697, 1.612) -- (0,2.8) -- (-1.697, 1.612) ;

 \filldraw[red] (0,0) circle (2pt);
 \filldraw[red] (-1.697, -1.612) circle (2pt);
\filldraw[red] (-2, -1.4) circle (2pt);
\filldraw[red] (-2, 1.4) circle (2pt);
\filldraw[red] (-1.697, 1.612) circle (2pt);

 \filldraw[red] (1.697, -1.612) circle (2pt);
\filldraw[red] (2, -1.4) circle (2pt);
\filldraw[red] (2, 1.4) circle (2pt);
\filldraw[red] (1.697, 1.612) circle (2pt);

\draw[blue, very thick] (-1.75, 1) .. controls (0, 1.75) .. (1.75,.5) node[midway, below]{\small $\kappa$};

 \filldraw[blue] (-1.75, 1) circle (2pt);
  \filldraw[blue] (1.75, .5) circle (2pt);

\end{axis}
\end{tikzpicture}
   \caption{Illustration of \emb{Proposition \ref{inner constant for box}}. On the top figure, any curve $\gamma$ tangent to $\mathcal K^s$ that intersects the center "diamond" must stay between the dotted lines. On the bottom figure, any curve $\kappa$ tangent to $\mathcal K^u$ with endpoints in the center region contained in $\tilde K(Q)$, must be contained within the dotted lines. So, one can choose a neighborhood $O$ of $Q$ accordingly.}
    \label{inner constant for box figure}

    \end{center}
\end{figure}
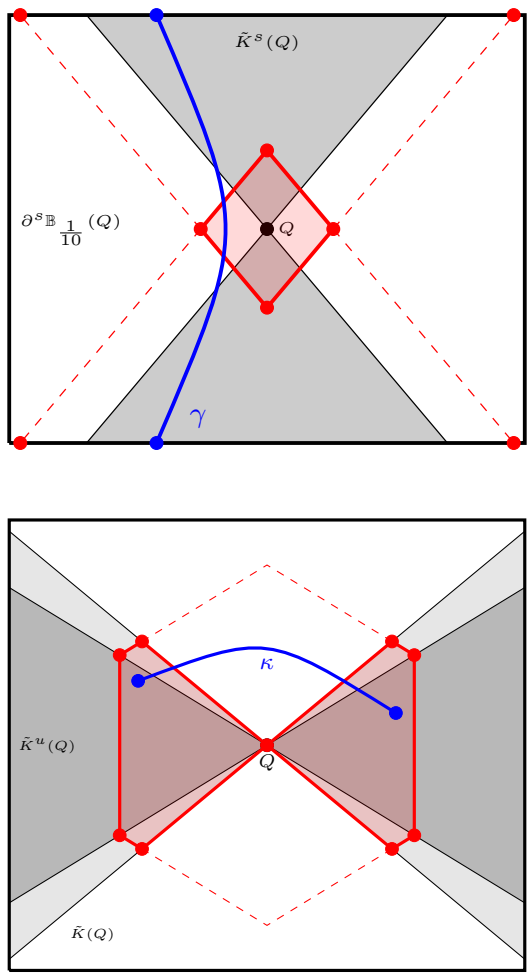
\emb{The next proposition is illustrated by Figure \ref{outer cone centering gives inner ball figure}.}
\emb{\begin{proposition}\label{outer cone centering gives inner ball}
    Suppose $K\supsetneq K^u$ is a cone disjoint from $K^s$, and $R$ is a rectangle centered at $Q$ with respect to $\tilde K(Q)$. Then, there is a neighborhood $O$ of $Q$ such that for any curve $\gamma$ tangent to $\mathcal K^u$, we must have that any component of $\gamma \cap R$ that intersects $O$ must not intersect $\partial^{ns} R$. 
\end{proposition}}

\begin{proof}
    \emb{The proof is analogous to that of Proposition \ref{inner constant for box}}.
\end{proof}
\emb{It should be noted that the size of the neighborhood $O$ in this proposition depends only on the cone $K$ and $\text{dist}(\partial R,Q)$.}

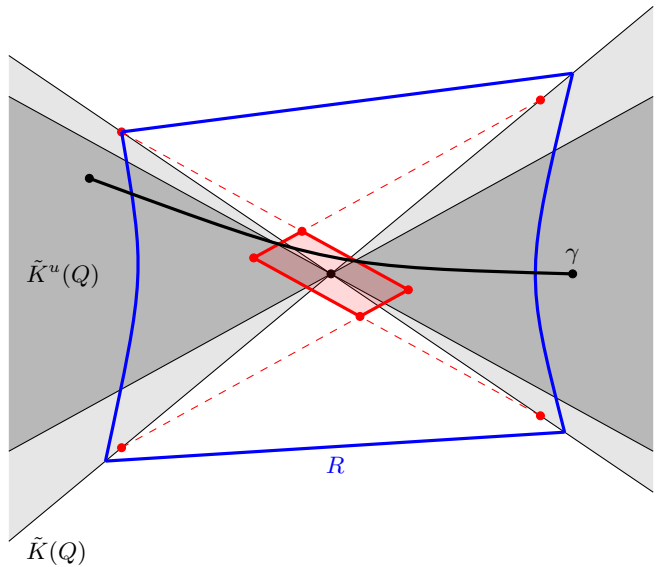
\begin{figure}
    \begin{center}
        \begin{tikzpicture}[scale=1.5]
\begin{axis}[xtick=\empty, ytick=\empty, axis line style={draw=none}, tick style={draw=none}]
    \addplot[name path=B, domain={-2:2}]{0.65*x};   
\addplot[name path=A, domain={-2:2}]{-0.65*x};  

 \addplot[name path=C, domain={-2:2}]{0.98*x};   
\addplot[name path=D, domain={-2:2}]{-0.8*x};  
       
\addplot[fill=black, fill opacity=0.2] fill between [of=A and B];

 \addplot[fill=black, fill opacity=0.1] fill between [of=C and D];

\filldraw[black] (0,0) circle (1pt);
\filldraw[red] (-1.3,0.98*-1.3) circle (1pt);
\filldraw[red] (1.3,0.98*1.3) circle (1pt);
\filldraw[red] (-1.3,-.8*-1.3) circle (1pt);
\filldraw[red] (1.3,-.8*1.3) circle (1pt);

\filldraw[red] (-0.18,0.312) circle (1pt);
\filldraw[red] (0.18,-0.312) circle (1pt);
\filldraw[red] (0.48,-0.117) circle (1pt);
\filldraw[red] (-0.48,0.117) circle (1pt);

 \draw[red, dashed] (-1.3,0.98*-1.3) -- (0.18,-0.312) ;
 \draw[red, dashed] (1.3,-.8*1.3) -- (0.18,-0.312) ;

 \draw[red, dashed] (-1.3,-.8*-1.3) -- (-0.18,0.312) ;

 \draw[red, dashed] (1.3,0.98*1.3) -- (-0.18,0.312) ;

 \draw[red, very thick]  (-0.18,0.312) -- (-0.48,0.117) -- (0.18,-0.312) -- (0.48,-0.117) --  (-0.18,0.312);

  \fill[red, opacity=0.15] (-0.18,0.312) -- (-0.48,0.117) -- (0.18,-0.312) -- (0.48,-0.117) --  (-0.18,0.312);

  \draw[blue, very thick] (-1.4,0.98*-1.4) .. controls (-1.15, 0) .. (-1.3,-.8*-1.3) ;
\draw[blue, very thick] (1.5,0.98*1.5) -- (-1.3,-.8*-1.3) ;

\draw[blue, very thick] (-1.4,0.98*-1.4) -- (1.45,-.8*1.45) node[midway, below]{\small $R$};

\draw[blue, very thick] (1.45,-.8*1.45) .. controls (1.2, 0) .. (1.5,0.98*1.5);

\coordinate[label=below right: \small $\tilde K(Q)$] (A) at (-1.95,-1.9*0.98);
 \coordinate[label=right: \small $\tilde K^u(Q)$] (B) at (-1.95,0);
\draw[black, very thick] (-1.5,.7) .. controls (0, 0.05) .. (1.5,0) node[above]{\small $\gamma$};

\filldraw[black] (-1.5,.7) circle (1pt);

\filldraw[black] (1.5,0) circle (1pt);
 
\end{axis}

\end{tikzpicture}

\caption{Illustration of \emb{Proposition} \ref{outer cone centering gives inner ball}. The rectangle $R$ is centered at $Q$ with respect to $\tilde K(Q)$ and has the property that for any curve $\gamma$ that is tangent to $\mathcal K^u$, any component of $\gamma\cap R$ that intersects the center quadrilateral must stay within the dotted lines and hence not intersect $\partial^{ns} R$.}
    \label{outer cone centering gives inner ball figure}

    \end{center}
\end{figure}

\section{Main Estimates and Non-Stationary Trace Map Dynamics}\label{Main Estimates and Non-Stationary Trace Map Dynamics}
\emb{We will now translate Theorem \ref{first intermediate theorem} into a problem concerning a sequence of Anosov maps of $\mathbb T^2$.} \emb{For the rest of this section}, let us consider an arbitrary finite set $\mathcal A\subset \mathbb N$ and an arbitrary sequence $\overline{a}:=(a_n)$ such that $a_n\in\mathcal A$ for all $n\in\mathbb N$. 

\emb{Recalling Lemma \ref{connection}, given $\rho>0$, we set $\hat L_{\lambda}\subset \mathbb T^2\setminus \tilde O_{r(\rho)}$ to be the curve given by $F_{\lambda}^{-1}\left( L_\lambda\setminus O_{\rho}\right)$.} \emb{Now, for appropriate range of $\lambda$, we define
    $$C(\lambda, O) :=\{P\in \hat L_{\lambda} \setminus O : \tilde T_{\lambda, a_n}\circ\cdots\circ \tilde T_{\lambda, a_1}(P)\not\in O \ \text{for all} \ n>0 \}.$$}
   \emb{Theorem \ref{first intermediate theorem} will follow from}
\emb{\begin{theorem}\label{intermediate theorem}
    For each $M>0$, there is a $\rho>0$ and $\lambda_0\in\left(0, \min_{a\in\mathcal A}\left\{\tilde \lambda(\rho, a)\right\}\right)$, such that 
    \begin{equation}\label{main estimate equation}
        \text{dim}_H\left(\mathcal C\left(\lambda, \tilde O_{r(\rho)}\right)\right)\geq \frac{\log 2}{\log\left(2+\frac{1}{M}\right)}
    \end{equation}
    for all $\lambda\in[0,\lambda_0]$.
\end{theorem}}

\emb{The proof of this theorem will be developed throughout Section \ref{Main Estimates and Non-Stationary Trace Map Dynamics} and will be concluded in Subsection \ref{concluding the proof of intermediate theorem}. Let us fix $M>0$.}

\subsection{\emb{Outline of Proof of Theorem \ref{intermediate theorem}}}
\emb{We outline the proof of the argument to come}

\begin{itemize}
    \item \emb{In Subsection \ref{Development of Constants}, we introduce a series of constants that will be used throughout the proof of this theorem, and these constants will determine $\rho$, and hence the size of the neighborhoods that make up $\tilde O_{r(\rho)}$.}
    
    \item \emb{Subsection \ref{Translation to a Toral Map Problem} is dedicated to choosing $\rho$ explicitly, and once chosen, in Subsection \ref{extended torus map}, we extend each map $\tilde T_{\lambda,a}$, for a small range of $\lambda$ and $a\in\mathcal A$, to a toral Anosov map, and these maps will satisfy a common cone condition.}

    \item \emb{In Subsection \ref{Properties of the Extended Toral Maps}, we then derive useful dynamical properties of these maps such as a bounded distortion estimate, and a non-stationary version of the Stable Manifold Theorem for these maps.}

    \item \emb{In Subsection \ref{Particular Rectangle Construction}, we construct families of sequences of dynamical rectangles containing each component of $\tilde O_{r(\rho)}$, that will be useful in controlling the dynamics of our maps near this open set. The stable sides of these rectangles will be composed of the stable manifolds constructed in Subsection \ref{Properties of the Extended Toral Maps}.}

    \item \emb{Lastly, in Subsection \ref{thickness estimate section}, for each $\lambda$ sufficiently near $0$, we construct a subset of $\mathcal C\left(\lambda,\tilde O_{r(\rho)}\right)$, that is roughly composed of the points whose orbits do not enter the interior of the constructed rectangles from Subsection \ref{Particular Rectangle Construction}. The proof is concluded in Subsection \ref{concluding the proof of intermediate theorem} by showing that this subset has large thickness (see Subsection \ref{Thickness of Cantor Sets Section} for a formal definition), and in particular, Hausdorff dimension greater than $\frac{\log 2}{\log\left(2+\frac{1}{M}\right)}$.}
\end{itemize}

\subsection{Development of Constants}\label{Development of Constants}

    We define a series of constants needed in the proof. Their significance will become clear as the proof progresses. 
    
Recalling (\ref{main cone and line}), we fix a cone $K:=K_{\beta_1}^u$, for some $\beta_1\in(\beta_0,0.1)$, so that it strictly contains $K^u$ and is transversal to $K^s$. We also choose $\beta_1$ close enough to $\beta_0$ so that the property described in \emb{Lemma} \ref{linear overflow} holds for the cone $K$. If $K(x,y)$ denotes the cone $K$ in the tangent space $T_{(x,y)}\mathbb T^2$, then we denote $\mathcal K:=\left\{K(x,y)\right\}_{(x,y)\in\mathbb{T}^2}$.\\
\\
We set
\begin{equation}\label{upper constant}
    C^*:=\frac{C^*_0}{C_1},
\end{equation}
where the constants $C_1$ and $C^*_0$ are taken from \emb{Lemma \ref{curve length compared to distance} and Proposition \ref{inner constant for box}}, respectively. Referring to Lemma \ref{linear cone condition}, we set 
$$\mu_1:=\overline \mu_{\beta_0}, \ \text{and} \  \mu_2:=\max_{n}\left\{\mu^u(a_n)\right\}.$$
    Furthermore, choose a $\delta>0$ such that 
    $$\mu_1-\delta>1,$$
   and choose $C_2>0$ such that  $$\|\hat T_{0,a}\|_{C^2}\leq C_2$$ for all $a\in\mathcal A$.
    
 \emb{\begin{definition}
    We say that a $C^1$ diffeomorphism $T:\mathbb T^2\rightarrow \mathbb T^2$, satisfies \textit{Property (C)} if for each $(x,y)\in \mathbb T^2$, we have 
    \begin{itemize}
    \item $D_{(x,y)}T^{-1}\left(K^s(x,y)\right)\subset \text{int}K^s\left( T(x,y) \right)$ and $D_{(x,y)} T\left(K^{u}(x,y)\right)\subset \text{int}K^u\left(T(x,y)\right)$
    \item $\left(\mu_1-\delta\right)|v|\leq \left|D_{(x,y)}T^{-1}v\right|\leq \left(\mu_2+\delta\right)|v|$ for $v\in K^s(x,y)$
    \item $\left(\mu_1-\delta\right)|v|\leq \left|D_{(x,y)}Tv\right|\leq \left(\mu_2+\delta\right)|v|$ for $v\in K^u(x,y)$
\end{itemize}
\end{definition}}

Certainly, all of our maps $\left\{\hat T_{0,a_n}\right\}$ satisfy Property (C) by Lemma \ref{linear cone condition}. We set 
\begin{equation}\label{geometric sum}
   C_3:=\sum_{n=-1}^{\infty}\frac{1}{(\mu_1-\delta)^{-(n+1)}}.
\end{equation}

\emb{\begin{lemma}\label{maximal angles}
There is a $\Theta>0$ such that the following holds: If $(T_n)$ is a sequence of maps satisfying Property (C) and $\sup_{n}\|T_n\|_{C^2}<C_2+1$, and if $\gamma$ is a smooth curve tangent to $\mathcal K^u$ so that $$\sum_{i=0}^{\emb{m}}\text{length}\left(T_i\circ\cdots \circ T_1(\gamma)\right)<C_1C_3\frac{\sqrt{2}}{10}$$ for \emb{some $m\geq 0$}, then $$\sum_{i=0}^{\emb{m}} \text{Ang}\left(T_{i}\circ\dots\circ T_1(\gamma) \right)<\Theta,$$ where $\text{Ang}(\kappa)$ denotes the supremum of the angles between tangent vectors of $\kappa$. 
\end{lemma}}
\begin{proof}
    \emb{Let $\gamma$ be a curve tangent to $\mathcal K^u$. From Property (C) and since the $C^2$ norms of each $T_n$ are uniformly bounded, the curvature of each curve $T_i\circ\dots\circ T_1(\gamma)$, $1\leq i\leq m$, is uniformly bounded. Hence, there is a constant $\Theta'>0$, that only depends on $C_2$ and the cones and constants in Property (C), such that 
    $$\text{Ang}\left(T_i\circ\dots\circ T_1(\gamma) \right)\leq \Theta'\text{length}\left(T_i\circ\dots\circ T_1(\gamma) \right)$$
    for each $1\leq i\leq m$. The lemma follows.}
\end{proof}
Let us define 
\begin{equation}\label{bounded distortion constant}
    C_4 :=\exp{\left[(C_2+1)^2\left(\Theta +C_1C_3\frac{\sqrt{2}}{10}\right)\right]}.
\end{equation}

\emb{Let us set $Q_0=\left(\frac{1}{4}, \frac{1}{4}\right)$}. By Lemma \ref{estimate for irrational lines} \emb{and Proposition \ref{inner constant for box}} we have 

    \begin{lemma}\label{upper constant and length}
        There is a $l^*>0$ such that if $L$ is a line through $Q_0$ of slope $-\frac{1}{[b_1, b_2,\dots]}$ with $b_k\in \mathcal A$, and $\text{length}(L)\geq l^*$, then for each $i=1,\dots, 4$, there are segments $L^+_i, \ L^-_i\subset L$ contained in $
        \mathbb B_i\setminus \partial^s\mathbb{B}_i$ such that 
    \begin{itemize}
        \item[1.)] $L^{\pm}_{i}$ connects the opposite lines that make up $\partial ^u\mathbb{B}_i$ and are centered at $Q_i$ with respect to $\tilde K(Q_i)$, where $L^{\pm}_{i}\cap \tilde K_{\pm}(Q_i)\neq \emptyset$
        \item[2.)] For any $P\in\left(L^{+}_{i}\cup L^{-}_{i}\right)\cap \tilde K(Q_i)$, we have $|P-Q_i|<\cos(\theta)\frac{C^*}{4}$
    \end{itemize}
where $\theta$ is the angle between boundary lines of $\tilde K(Q_i)$ (\emb{explicitly, $\theta=\theta_0(\beta_1)$, see discussion before (\ref{main cone and line})}).
    \end{lemma}
    
\begin{proof}
    \emb{Let $i\in\{1,2,3,4\}$. By Lemma \ref{estimate for irrational lines}, we can choose a $l^*>0$ such that if $L$ is a line through $Q_0$ of slope $\frac{-1}{[b_1,b_2,\dots]}$, $b_i\in \mathcal A$, and $\text{length}(L)\geq l^*$, then there is a segment $L'_i\subset L$ such that
    \begin{itemize}
        \item $L'_i\subset \tilde K_{+}(Q_i)$ and connects the boundaries of $\tilde K_{+}(Q_i)$
        \item For each $P\in L_i'$, we have $|P-Q_i|\leq \min\left\{C_0^*, \ \cos(\theta)\frac{C^*}{4}\right\}.$
    \end{itemize}
    By increasing $l^*$ if needed, we can even assume that $L'_i$ connects the boundary lines of $\tilde K_{+}(Q_i)$.\\
    \\
    Let $L$ be such a line and let $L^+_i$ be the longest segment of $L\cap \overline{\mathbb B}_i$ that contains $L'_i$.
    Then, by Proposition \ref{inner constant for box}, since 
    $$\text{dist}\left(L_i^+, Q_i\right)<C_0^*,$$
    we have that $L_i^+\cap \partial^s\mathbb B_i=\emptyset.$ By increasing $l^*$ again if needed, we can also ensure that $L_i^+$ connects the components of $\partial^u\mathbb B_i$ and Proposition \ref{inner constant for box} will still guarantee that this extended segment is disjoint from $\partial^s\mathbb B_i$. By repeating this argument for each $i$ and with $+$ replaced with $-$, the lemma is proved.}
\end{proof}

Let us choose a $\tilde n$ so that $\frac{\left(\mu_1-\delta\right)^{\tilde n}}{10}>l^*$ and set $l^{**}:=\left(\mu_2+\delta\right)^{\tilde n+1}$. \emb{For $m\geq 1$, denote $$\text{Orb}_0\left(Q_0\right):=\left\{\hat T_{0,a_{n+m}}\circ\cdots\circ \hat T_{0,a_m}\left(Q_0\right)\right\}_{n\geq 0}\cup \left\{Q_0\right\}.$$
It is a calculation to check that $\text{Orb}_0\left(Q_0\right):=\bigcup_{m\geq 0} \text{Orb}_0\left(Q_0\right)$ is a finite set and has distance from each $Q_i$ that is at least $\frac{1}{4}$.} Again by Lemma \ref{estimate for irrational lines}, we have 

\begin{lemma}\label{lower constant and length}
    There is a $C^{**}\in\left(0, \frac{1}{4}-\frac{1}{10}\right)$ such that if $L$ is a line through \emb{$Q$, for some $Q\in \text{Orb}_0\left(Q_0\right)$}, of slope $-\frac{1}{[b_1, b_2,\dots]}$ with $b_k\in \mathcal A$, and $\text{length}(L)\leq l^{**}$, then 

\begin{equation}
    \text{dist}(L, Q_i)\geq C^{**},
\end{equation}
for each $i=1,\dots, 4$.
\end{lemma}
\begin{proof}
    \emb{The proof is analogous to that of Lemma \ref{upper constant and length} but through an application Lemma \ref{estimate for irrational lines} part (2). The quantity $\frac{1}{4}-\frac{1}{10}$ comes from the distance between $\text{Orb}_0\left(Q_0\right)$ and each $Q_i$, and the size of each $\mathbb B_i$.}
\end{proof}
    Recalling our constant $C_1$ from \emb{Lemma} \ref{curve length compared to distance}, we will choose $N^*\in\mathbb N$ such that 
    \begin{equation}\label{N large enough}
        C_4^{-1}\frac{\frac{C^{**}}{4}}{C_1 C^*\left(\mu_1-\delta\right)^{-N^*}}-C_4^{-1}\\
    \geq M
    \end{equation}
   and
   \begin{equation}\label{Rectangle not too big}
       C_1(\mu_1-\delta)^{-N^*}C^*\leq \frac{C^{**}}{4}.
   \end{equation}

   \subsection{ \emb{Choice of Neighborhood}}\label{Translation to a Toral Map Problem}
We are now ready to choose the neighborhood $O$ that was described at the beginning of this section.

   From \emb{Proposition} \ref{outer cone centering gives inner ball}, there is a  $r_0=r_0\left(N^*, C^{**}, K\right)$ such that

\begin{equation}\label{fixed inner ball radius}
    r_0<(\mu_2+\delta)^{-N^*-1}\frac{C^{**}}{4},
\end{equation}
and for each $Q\in\mathbb T^2$ and any rectangle $R$ centered at $Q$ with respect to $\tilde K(Q)$ satisfying

\begin{itemize}
    \item $\partial^s_{\bullet} R $ connects opposite boundaries of $\tilde K_{\bullet}(Q)$ for each $\bullet\in\{+,-\}$
    \item $\text{dist}\left(\partial^s R, Q\right)\geq (\mu_2+\delta)^{-N^*-1}\frac{C^{**}}{2}$
    \item $\partial ^{ns} R$ is a union of two disjoint lines with endpoints belonging to $\partial \tilde K(Q)$
\end{itemize}
we must have that $B_{r_0}(Q)\subset R$ and if $\gamma$ is a curve tangent to $\mathcal K^u$ that connects $\partial^s_{+}R$ and $\partial^s_{-}R$ and satisfies $\gamma \cap B_{r_0}(Q)\neq \emptyset$, then $\gamma$ is completely contained in $R$ and does not intersect $\partial^{ns} R$.

Referring to Lemma \ref{connection}, let us choose $\rho>0$ so small that $r(\rho)<\frac{r_0}{2}$ and set $\lambda_1:=\min_{a\in\mathcal A}\left\{\tilde \lambda(\rho, a)\right\}$. 

\emb{For $\lambda\in(0,\lambda_1)$}, it is a straightforward calculation to check that the line $\hat L_0:=\left[0,\frac{1}{2}\right]\times\{0\}\subset \mathbb T^2$ is mapped into $L_{0}\cap \mathbb S$ under the map $F$. Also, the projection map $\pi_{\lambda}$ sends $L_{\lambda}\cap \mathbb S_{\lambda, O_{\rho}}$ onto a curve in $\mathbb S$ that is $C^2$ close to a segment of $L_0$, provided $\lambda$ is small. \emb{That is, for small $\lambda$, the curve $\hat L_{\lambda}\subset \mathbb T^2\setminus \emb{\tilde O_{r(\rho)}}$ is $C^2$ close to a segment of $\hat L_0$, and in particular, $\hat L_0\setminus \emb{\tilde O_{r(\rho)}}$ as $\lambda\rightarrow 0$ in the $C^2$-sense.}

\emb{So, to complete the proof of Theorem \ref{intermediate theorem}, we will show} that there is a $\lambda_0\in(0,\lambda_1)$ such that 

\begin{equation}\label{Hausdorff dimension of torus Cantor}
        \text{dim}_{H}\left( \mathcal C\left(\lambda, \emb{\tilde O_{r(\rho)}}\right)\right)\geq \frac{\log 2}{\log\left(2+\frac{1}{M}\right)}
    \end{equation}
    for all $\lambda\in[0,\lambda_0]$.
    
    \subsection{\emb{Extended Toral Maps}}\label{extended torus map}

Now, for each $a\in\mathcal A$ and $\lambda\in [0,\lambda_1]$, we will construct an Anosov diffeomorphism that coincides with $\tilde T_{\lambda, a}$ outside of $O_{r_0}$.

\begin{lemma}\label{extension}
For each $a\in\mathcal A$ and $\lambda\in\left[0,\lambda_1\right]$, there is a smooth diffeomorphism $\hat T_{\lambda,a}:\mathbb{T}^2\rightarrow \mathbb{T}^2$,  such that $\hat T_{\lambda,a}=\tilde T_{\lambda, a}\restriction_{\mathbb{T}^2\setminus O_{r_0}}$ and $\lim\limits_{\lambda\rightarrow 0}\left\|  \hat T_{\lambda, a} -\hat T_{0,a} \right\|_{C^2}=0$ on $\mathbb{T}^2$.
\end{lemma}

\begin{proof}
   Let $\psi:\mathbb T^2\rightarrow [0,1]$ be a smooth function such that $\psi=0$ on $\mathbb T^2\setminus O_{r_0}$ and $\psi=1$ on $O_{\frac{r_0}{2}}$. Then, as $\tilde O_{r(\rho)}\subset O_{\frac{r_0}{2}}$, we must have that $\tilde T_{\lambda, a}$ is defined on $\mathbb T^2\setminus O_{\frac{r_0}{2}}$.
   
   Writing in components 
   $$\tilde T_{0,a}(x,y)=\left(\tilde T_{0, a}^1(x,y),\tilde T_{0, a}^2(x,y)\right)  \ \text{and} \ \tilde T_{\lambda,a}(x,y)=\left(\tilde T_{\lambda, a}^1(x,y),\tilde T_{\lambda, a}^2(x,y)\right) $$
   where $\tilde T_{0,a}^1(x,y)=ax+y$, $\tilde T_{0,a}^2(x,y)=x$, and $\tilde T_{\lambda, a}^1(x,y), \ \tilde T_{\lambda, a}^2(x,y):\mathbb T^2\setminus \emb{\tilde O_{r(\rho)}}\rightarrow S^1$, we set
   $$\hat T_{\lambda, a}(x,y):=\left(\hat T_{\lambda, a}^1(x,y), \hat T_{\lambda, a}^2(x,y)\right),$$
   where 
   $$\hat T_{\lambda, a}^1(x,y):=\psi(x,y) \tilde T_{0, a}^1(x,y)+(1-\psi(x,y))\tilde T_{\lambda, a}^1(x,y)$$
   and 
   $$\hat T_{\lambda, a}^2(x,y):=\psi(x,y) \tilde T_{0, a}^2(x,y)+(1-\psi(x,y))\tilde T_{\lambda, a}^2(x,y).$$
   The result follows from computing all partial derivatives up to order 2 of $\hat T_{\lambda, a}$, and comparing them to those of $\hat T_{0, a}$.
\end{proof}
To ease notation, we will denote $f_{\lambda, n, m}:=\hat T_{\lambda, a_{n+m}}\circ\cdots \circ \hat T_{\lambda, a_{m}}$, for $n,m\geq 0$, where $\hat T_{\lambda,0}:=\text{id}$.

\emb{\begin{lemma}\label{common orbit}
    For each $\lambda\in\left[0,\lambda_1\right]$, we have 
    \begin{equation}\label{common orbit equation}
        \hat T_{\lambda, a_{n+m}}\circ\cdots\circ \hat T_{\lambda, a_m}\left(\frac{1}{4}, \frac{1}{4}\right)=\hat T_{0, a_{n+m}}\circ\cdots\circ \hat T_{0, a_m}\left(\frac{1}{4}, \frac{1}{4}\right)
    \end{equation}
    for any $m\geq 0$ and $n\in\mathbb N$.
\end{lemma}
\begin{proof}
    Based on the compositions from (\ref{auxiliary map decomposition}) and construction of $\pi_{\lambda}$ in Lemma \ref{connection}, one can check that for any $\lambda\in[0,\lambda_1]$, the set $$\left\{ \left(T_{a_n}\circ\cdots\circ T_{a_1}\right)\left(-\sqrt{1+\frac{\lambda^2}{4}},0,0   \right) : n\in\mathbb N\right\}$$ is finite and
    $$\left(\pi_{\lambda}\circ T_{a_n}\circ\cdots\circ T_{a_1}\right)\left(-\sqrt{1+\frac{\lambda^2}{4}},0,0   \right)=T_{a_n}\circ\cdots\circ T_{a_1}(-1,0,0).$$ Since $F\left(\frac{1}{4}, \frac{1}{4}\right)=(-1,0,0)$, and $\left(\frac{1}{4}, \frac{1}{4}\right)\in \mathbb T^2\setminus O_{r_0}$ (see \ref{fixed inner ball radius}), this gives 
    $$\hat T_{\lambda, a_n}\circ\cdots\circ \hat T_{\lambda, a_1}\left(\frac{1}{4}, \frac{1}{4}\right)=\hat T_{0, a_n}\circ\cdots\circ \hat T_{0, a_1}\left(\frac{1}{4}, \frac{1}{4}\right)$$
    for any $n\in\mathbb{N}$, and this non-stationary orbit is a finite set. This same argument gives (\ref{common orbit equation}) for any $m\geq 0$.
\end{proof}}

\emb{\subsubsection{Further Choice of Constants.}}
 Now, one may choose a $\lambda_2\in(0,\lambda_1)$ such that for any $\lambda\in [0,\lambda_2]$ and $n\in\mathbb N$, we have 
 \begin{itemize}
     \item  The map $\hat T_{\lambda, a_n}$ satisfies Property (C) 
     
     \item We have the estimate 
     \begin{equation}\label{C^2 norm bound}
         \left\|\hat T_{\lambda, a_n}\right\|_{C^2}<C_2+1
     \end{equation}
     
     \item For any $P\in\mathbb {T}^2$, we have 
     \begin{equation}\label{Orientation for toral maps}
         D_P\hat T_{\lambda, a_n} \left(K_{\bullet} (P) \right)\subset K_{\bullet}\left(\hat T_{\lambda, a_n}(P)\right)
     \end{equation}
     for each $\bullet\in\{+,-\}$.

     \item The smooth curve $\hat L_{\lambda}$ is tangent to $\mathcal K^u$ 
 \end{itemize}
 The first and second part can be taken as a consequence of Lemma \ref{extension}. The third part comes from the fact that $A_a$ has a positive unstable eigenvalue for each $a$, so if $\hat T_{\lambda, a}$  is sufficiently $C^1$-close to $\hat T_{0, a}$, then the differential of $\hat T_{\lambda, a}$ will preserve orientation of vectors from the cone $K$. The last part comes from the fact that the curve $\hat L_\lambda$ tends to $\hat L_0\setminus \emb{\tilde O_{r(\rho)}}$ in the $C^2$-sense.

\subsection{\emb{Properties of the Extended Toral Maps}\label{Properties of the Extended Toral Maps}}
\emb{This subsection is dedicated to deriving useful dynamical properties of our extended maps $\left\{ \hat T_{\lambda, a} \right\}$.}
\emb{\subsubsection{Lengths and Orientation of  Curves}\label{Lengths an Orientation of Curves Segment}
From Property (C), we have
\begin{lemma}\label{lengths of curves}
For all $\lambda\in[0,\lambda_2]$, $m\in\mathbb{N}$ and $n\geq 0$, we have
\begin{itemize}
    \item[1.)] If $\gamma$ is a curve tangent to $\mathcal K^s$, then $$(\mu_1-\delta)^{n+1}\text{length}(\gamma)\leq\text{length}\left(f_{\lambda, n, m}^{-1}(\gamma)\right)\leq (\mu_2+\delta)^{n+1}\text{length}(\gamma)$$
    \item[2.)] If $\gamma$ is a curve tangent to $\mathcal K^u$, then $$(\mu_1-\delta)^{n+1}\text{length}(\gamma)\leq\text{length}\left(f_{\lambda, n, m}(\gamma)\right)\leq (\mu_2+\delta)^{n+1}\text{length}(\gamma)$$
\end{itemize}
\end{lemma}}
\emb{\begin{remark}\label{lifted curve orientation}
   Suppose that $\kappa$ is a curve contained in $\mathbb B_{\frac{1}{10}}(Q)$, for some $Q\in\mathbb T^2$, that is tangent to $\mathcal K$, with one of its endpoints being $Q$. Then, we must have that $\kappa\subset  \tilde K_{\bullet}(Q)$ for $\bullet=+$ or $-$. If we parameterize $\kappa$ via $\kappa(t):[0,1]\rightarrow \mathbb{T}^2$ so that it has positive orientation, then for each $t\in[0,1]$, we would have $\dot \kappa(t)\in K_{+}\left(  \kappa(t) \right)$, and hence (\ref{Orientation for toral maps}) implies that $f_{\lambda, n, m}( \kappa)\subset\tilde K_{\bullet}\left(f_{\lambda, n, m}(Q)\right) $, where one of its endpoints is $f_{\lambda, n, m}(Q)$. In other words, the orientation of curves tangent to $\mathcal K$ is preserved under the map $f_{\lambda, n, m}$.
\end{remark}}

\subsubsection{Bounded Distortion Estimate}
If $P$ and $Q$ are points on a curve $\gamma$, we denote the length of the segment of $\gamma$ from $P$ to $Q$ by $\text{dist}_{\gamma}(P,Q)$. Recalling (\ref{bounded distortion constant}), from Theorem 2 of \cite{BM}, we obtain the following bounded distortion estimate:
\begin{lemma}\label{bounded distortion}
    For each $\lambda\in[0,\lambda_2]$, if $\gamma$ is a sub-curve of $\hat L_\lambda$ that has endpoints $P$ and $Q$ and satisfies $f_{\lambda, m,0}(\gamma)\in \mathbb B_i$ for some $i=1,\dots, 4$ \emb{and $m\geq 0$}, then for any point $Q'\in \gamma\setminus\{P,Q\}$, we must have 
    \begin{align*}
        C_4^{-1}\frac{\text{dist}_{\gamma}(P,Q')}{\text{dist}_{\gamma}(Q,Q')}&\leq \frac{\text{dist}_{f_{\lambda, m, 0}(\gamma)}\left(f_{\lambda, m, 0}(P),f_{\lambda, m, 0}(Q')\right)}{\text{dist}_{f_{\lambda, m, 0}(\gamma)}\left(f_{\lambda, m, 0}(Q),f_{\lambda, m, 0}(Q')\right)}\\
        &\leq C_4\frac{\text{dist}_{\gamma}(P,Q')}{\text{dist}_{\gamma}(Q,Q')}.
    \end{align*}
\end{lemma}

\begin{proof}
Let $\lambda\in[0,\lambda_2]$. \emb{If} $f_{\lambda, m,0}(\gamma)\in \mathbb B_i$, it must have length less than $C_1\frac{\sqrt{2}}{10}$. Thus, from Lemma \ref{lengths of curves} and recalling (\ref{geometric sum}), it must be the case that for each $k$ with $0\leq k\leq m$, we have
$$\text{length}\left(f_{\lambda, k, 0}(\gamma)\right)\leq (\mu_1-\delta)^{-(m-k)}\text{length}(f_{\lambda, m, 0}(\gamma))\leq C_1\frac{\sqrt{2}}{10}(\mu_1-\delta)^{-(m-k)},$$
so that
$$\sum_{k=0}^{m}\text{length}(f_{\lambda,k,0}(\gamma))\leq C_1\frac{\sqrt{2}}{10}\sum_{k=0}^{m} (\mu_1-\delta)^{-(m-k)}\leq C_1C_3\frac{\sqrt{2}}{10}.$$
 
In addition, since we have that $\hat L_{\lambda}$ is a smooth curve tangent to $\mathcal K^u$ and $\left \|\hat T_{\lambda, a_n} \right\|\leq C_2+1$ for all $n\in\mathbb N$, it follows from \emb{Lemma} \ref{maximal angles} that 
$$\sum_{k=0}^m\text{Ang}\left(f_{\lambda, k, 0}\emb{\left(\gamma\right)}\right)<\Theta.$$
 Thus, using Theorem 2 of \cite{BM}, the result holds for $C_4$ as defined in (\ref{bounded distortion constant}).
\end{proof}

\subsubsection{\emb{Non-Stationary Stable Manifolds}}\label{stable manifold section}
In this \emb{subsection}, we establish a non-stationary version of the Stable Manifold Theorem for our collection of maps.

Consider the set
$$W^s_{\lambda}\left(Q, m\right):=\left\{P \in \mathbb T^2: \lim\limits_{n\rightarrow \infty} \emb{\left|f_{\lambda,n,m}(P)- f_{\lambda,n,m}(Q)  \right|}=0 \right\},$$
which we call the \textit{non-stationary stable manifold} of $\left(T_{a_{n+m}}\right)_{n\geq 1}$ at $Q$. For $\eta>0$, we define 
$$W^s_{\lambda, \text{loc}}(Q,m):=\{P\in B_{\eta}(Q) : f_{\lambda, n, m}(P)\in B_{\eta}(f_{\lambda, n, m}(Q)) \ \text{for all} \ n\geq 0\}.$$
We will denote 
\begin{equation}
    \text{Orb}_{\lambda, m}(Q):=\{f_{\lambda,n,m}(Q)\}_{n\geq 0}\cup\{Q\} \ \text{and} \ \text{Orb}_{\lambda}(Q):=\bigcup_{m=1}^{\infty}\text{Orb}_{\lambda,m}(Q).
\end{equation}
\begin{theorem}\label{stable manifold theorem}
There is a $\lambda_3\in(0,\lambda_2)$ such that for all $\lambda\in[0,\lambda_3]$ and $\eta=\frac{1}{10}$, we have
\begin{itemize}
   \item[1.)] For all $Q\in\mathbb T^2$ and $n\in\mathbb{N}$, we have that $W^s_{\lambda, \text{loc}}(Q,n)$ is a $C^1$ curve that is tangent to $\mathcal K^s$ and satisfies 
   $$\hat T_{\lambda, a_n}\left(W^s_{\lambda, \text{loc}}(Q,n)  \right)\supset W^s_{\lambda, \text{loc}}\left(\hat T_{\lambda, a_n}(Q), n+1    \right).$$
\item[2.)] For each $\epsilon>0$, there is a $\lambda'(\epsilon)>0$ such that if $\lambda<\lambda'(\epsilon)$, then 
$$\emb{\text{dist}_H} \left( W^s_{0,\text{loc}}(Q, k), W^s_{\lambda, \text{loc}}(Q, k)\right)<\epsilon$$
for all $k\in\mathbb{N}$ and any point $Q$ satisfying 
$f_{\lambda, n, k}(Q)=f_{0,n,k}(Q)$ for any $n, k \geq 0$ and each $\lambda$.

\item[3.)] We have that $W^s(Q,k)=\bigcup_{n=0}^{\infty} f_{\lambda, n, k}^{-1}\left(W^s_{\lambda, \text{loc}} (f_{\lambda, n, k}(Q), k+n+1) \right)$ and there is a $C^1$ injective immersion $\iota_{\lambda, Q,k}:\mathbb{R}\rightarrow \mathbb{T}^2$ such that $\iota_{\lambda, Q,k}(\mathbb{R})=W^s_{\lambda}(Q,k)$ and $TW^s_{\lambda}(Q,k)\subset \mathcal K^s$.
\end{itemize}
   
\end{theorem}
Parts (1) and (3) of this Theorem are very close to a Corollary of Theorem 6.2.8 from \cite{KH}.

We will sketch the proof of these items using similar techniques from \cite{KH}, in our context. 
\begin{proof}[Sketch of Proof of Theorem \ref{stable manifold theorem}.]
We prove the result first for $Q=0$. Recalling our notations from Section \ref{local cones and boxes notation}, we work locally in the neighborhood $\mathbb B_1=\pi\left(\text{Box}_{\frac{1}{10}}(0)\right)$ and in the coordinate system $\tilde V^\perp\oplus \tilde V$ where $\tilde V=\pi\left(V\cap \text{Box}_{\frac{1}{10}}(0)\right)$

Notice that there is a $l<1$ such that $\varphi:\tilde V^{\perp}\rightarrow  \tilde V$ is tangent to $\mathcal K^s$ if and only if $\text{Lip}(\varphi)<l$.

We set $$X:=\left\{ \text{graph}(\varphi) : \varphi\in \text{Lip}\left( \tilde V ^{\perp},  \tilde V\right), \ \text{Lip}(\varphi)<l, \ \varphi(0)=0 \right\}$$ which is complete under the metric $d_X$ via
$$d_X\left(\text{graph}(\varphi), \text{graph}(\varphi')\right):=d_{C^0}(\varphi, \varphi').$$

From \emb{Lemma} \ref{linear overflow} and the cone conditions, we may choose a $\lambda_3>0$ such that for each $\lambda\in[0,\lambda_3]$ and each $a\in\mathcal A$, if $\text{graph}(\varphi)\in X$, then its image under $\hat T_{\lambda,a}^{-1}$ must overflow $\mathbb B_1$ and be tangent to $\mathcal K^s$.

For each $\lambda\in[0,\lambda_3]$ and $n\in\mathbb{N}$, let us define $\Gamma_{\lambda, a_n}: X\rightarrow X$ so that $\Gamma_{\lambda, a_n}\left(\text{graph}(\varphi)\right)$ is the segment of $\hat T_{\lambda, a_n}^{-1}\left(\text{graph}(\varphi)\right)\cap \mathbb B_1$ that contains $0$. By the previous discussion, this map is well-defined. 

Through the common cone conditions (Property (C)), for each $\lambda\in[0,\lambda_3]$, we find that $\left( \Gamma_{\lambda, a_n} \right)$ is a sequence of contractions with uniform contraction rate over $n$. Thus, by Proposition \ref{nonstationary contractions}, for each $k\in\mathbb N$, there is a $\varphi_{\lambda, k}^*: V^{\perp}\rightarrow V$ such that $\text{graph}\left(\varphi_{\lambda,k}^*\right)\in X$ and 
$$\lim\limits_{n\rightarrow\infty}\Gamma_{\lambda, a_k}\circ\cdots\circ \Gamma_{\lambda, a_{n+k}}\text{graph}\left(\varphi\right)=\text{graph}\left(\varphi_{\lambda,k}^*\right)$$
for any $\text{graph}\left(\varphi\right)\in X$. By construction, it follows that 
$$W^s_{\lambda, \text{loc}}\left(0, k\right)=\text{graph}\left(\varphi^*_{\lambda, k}\right).$$

Again, through the cone conditions and $C^1$ smoothness of our maps, one finds that $\varphi_{\lambda,k}^*$ is differentiable and that
$$T_{P}\text{graph}\left( \varphi_{\lambda,k}^*\right)=\bigcap_{k=1}^{\infty}  D_{f_{\lambda, n, k}(P)} \left(f_{\lambda, n, k}\right)^{-1}\left[K^s\left( f_{\lambda, n, k}(P) \right)\right],$$
for each $P\in\text{graph}\left( \varphi_{\lambda,k}^*\right)$ where the resulting set is a one-dimensional line in $K^s(P)$, and these lines vary continuously in $P$ (see Proposition 6.2.12 and Lemma 6.2.15 from \cite{KH}). Hence, $\varphi_{\lambda,k}^*$ must be $C^1$ with its graph being tangent to the cone field $\mathcal K^s$. This establishes (1) and also (3).

To establish (2), we notice that the distance between $\Gamma_{0,n}$ and $\Gamma_{\lambda,n}$ in $(X,d_X)$, can be made arbitrarily small for a sufficiently small range of $\lambda$ due to the fact that $\hat T_{\lambda, a_n}$ is a $C^2$ perturbation of $\hat T_{0, a_n}$, and this smallness can be made uniform in $n$.

Hence, by Proposition \ref{nonstationary contractions}, we can make the corresponding fixed points $\text{graph}\left(\varphi_{0,k}^*\right)$ and $\text{graph}\left(\varphi_{\lambda,k}^*\right)$ arbitrary close in the metric $d_X$ for sufficiently small $\lambda$, and this closeness can be made uniform in $k$. 

For the case $Q\neq 0$, we repeat the process but for the family of maps given by 
$$P\mapsto \hat T_{\lambda, a_n}\left( P+f_{\lambda, n-1, 0}(Q)\right)-f_{\lambda, n-1, 0}(Q).$$
\end{proof}
Notice that in the case $\lambda=0$, we have $T_QW^s_0(Q,m)=E^s_{[a_m, a_{m+1},\dots]}$, but since $W^s_0(Q,m)$ is a line, from Lemma \ref{slope of linear stable manifolds}, we deduce that it is precisely the line through $Q$ of irrational slope given by $\frac{-1}{[a_m, a_{m+1},\dots]}.$

\begin{remark}
    One could claim higher regularity of $W^s_{\lambda,\text{loc}}(Q,k)$ and in fact show that these curves are smooth or even real analytic, but we only provide the statement we actually need for the proof.
\end{remark}

 The next Corollary expands on (2) from the previous theorem.
\begin{corollary}\label{perturbations of stable manifolds}
    Suppose that $Q_0\in\mathbb{T}^2$ is a point such that 
    $f_{\lambda, n, k}(Q_0)=f_{0,n,k}(Q_0)$ for any $n, k \geq 0$ and each $\lambda>0$, and that $\text{Orb}_0(Q_0)$ is finite. 
    For each $\epsilon>0$ and $n_0\in\mathbb{N}$, there is a $\lambda'\in \left(0,\lambda_3\right)$ such that if $\lambda\in\left[0, \lambda'\right]$, we have that $$\sup\limits_{\substack{k\in\mathbb{N} \\ 0\leq n \leq  n_0 \\ m\in\mathbb{N}}}\emb{\text{dist}_H}\left( f^{-1}_{\lambda, n, m} W^s_{\lambda, \text{loc}}(Q, k),  f^{-1}_{0, n, m} W^s_{0, \text{loc}}(Q, k)\right)<\epsilon,$$
    for all $Q\in\text{Orb}_{\lambda}(Q_0)$.
\end{corollary}
\begin{proof}
Since for each $\lambda$, the collection of maps $\bigcup_{n=1}^{n_0} \bigcup_{m\in\mathbb{N}} \left\{f^{-1}_{\lambda,n,m}\right\}$ is finite, it is equicontinuous. In addition, since the map $g\mapsto g^{-1}$ is continuous over the space $\text{Diff}^{2}\left(\mathbb T^2\right)$, since each map $\hat T_{\lambda, a_n}$ is a $C^2$ perturbation of $\hat T_{0,a_n}$, and again since our collection of maps is finite, we can make $f_{\lambda, n, m}^{-1}$ and $f_{0,n,m}^{-1}$ arbitrarily close in the $C^0$ norm, where the closeness is uniform in $m\in\mathbb N$ and $n\leq n_0$, for a sufficiently small range of $\lambda$. Together with (2) of Theorem \ref{stable manifold theorem}, the result follows.
\end{proof}
We note that while the result above is true for arbitrary $k$, the sets only have dynamical meaning for certain values of $k$. What we notice is that starting with any $m$, we have the nested sets
\begin{align*}
    &W_{\lambda,\text{loc}}^s(Q,m)\\
    &\subset \hat T_{\lambda, a_m}^{-1}W_{\lambda,\text{loc}}^s\left(\hat T_{\lambda, a_{m}}(Q),m+1\right)=f_{\lambda,0,m}^{-1}W^s_{\lambda,\text{loc}}(f_{\lambda,0,m}(Q), m+1)\\
    &\subset \hat T_{\lambda, a_m}^{-1}\circ \hat T_{\lambda, a_{m+1}}^{-1} W_{\lambda,\text{loc}}^s\left(\hat T_{\lambda, a_{m+1}}\circ \hat T_{a_m}^{-1} (Q),m+2\right)=f^{-1}_{\lambda,1,m}W^s(f_{\lambda,1,m}(Q),m+2)\\
    &\subset \cdots\\
    &\subset \hat T_{\lambda, a_{m}}^{-1}\circ\cdots \circ \hat T_{\lambda, a_{m+n_0}}^{-1}W_{\lambda,\text{loc}}^s\left(\hat T_{\lambda, a_{m+ n_0}}\circ\cdots\circ\hat T_{\lambda, a_{m+1}}(Q),m+n_0+1\right)\\
    &=f^{-1}_{\lambda,n_0,m}W^s(f_{\lambda,\tilde n,m}(Q),m+n_0+1)\\
    &\subset \cdots\\
    &\subset W^s_{\lambda}(Q,m)
\end{align*}
and Corollary \ref{perturbations of stable manifolds} gives that for a small range of $\lambda$, depending on $n_0$, these sets are $\epsilon$-close to the corresponding sets for $\lambda=0$, and this is true for any $Q\in \text{Orb}(Q_0)$. That is, arbitrarily long pieces of the stable manifolds $W^s_{\lambda}(Q,m)$ containing $Q$ can be made arbitrarily close to pieces of $W^s_{0}(Q,m)$ containing $Q$, uniformly in $m$, provided the pieces are of uniform size, for a sufficiently small range of $\lambda$.

\subsection{Rectangle Construction}\label{Particular Rectangle Construction}
\emb{We will now construct a specific family of dynamical rectangles that contain each neighborhood of $O_{r_0}$.}
Recall that by \emb{Lemma} \ref{common orbit}, for $Q_0=\left(\frac{1}{4},\frac{1}{4}\right)$, we have that 
$$f_{\lambda, n, m}(Q_0)=f_{0, n, m}(Q_0)$$
for each $n,m\geq 0$ and $\lambda\in[0,\lambda_2]$, and \emb{as pointed out before Lemma \ref{lower constant and length}}, $\text{Orb}_{0}(Q_0)$ is a finite set and bounded away from each $Q_i$ by at least $\frac{1}{4}$. 

Now, for each $i=1,\dots, 4$, $m\in\mathbb N$, and $\lambda\in[0,\lambda_3]$, we will construct a family of rectangles containing $Q_i$ that have stable sides belonging to $W^s_\lambda(Q_0, m)$ and also have controlled size.

The idea is that for each $i=1,\dots, 4$ and $m\in \mathbb N$, we will find a rectangle $R$ that contains $B_{r_0}(Q_i)$ and has stable sides belonging to $W^s_{\lambda}(Q_0, m+1)$, and such that any segment from $\hat L_{\lambda}$ that intersects $B_{r_0}(Q_i)$, under the map $f_{\lambda, m, 0}$,  must also intersect the stable sides of the rectangle $R$. This means that for future iterates, the points that intersected the stable boundaries of $R$ will move away from the points $Q_i$.

This dynamical behavior will allow us to construct a Cantor set contained in $\mathcal C\left(\lambda, \emb{\tilde O_{r(\rho)}}\right)$ that has \emb{Hausdorff dimension greater than $\frac{\log 2}{\log\left(2+\frac{1}{M}\right)}$.}
\subsubsection{Base Set of Rectangles}

The construction of our rectangles will be done in an inductive process, so we start by constructing a base set of rectangles.

Recalling the constant $C^*$ from (\ref{upper constant}) and $C^{**}$ from Lemma \ref{lower constant and length}, we have (see Figure \ref{Outer Rectangle Figure})

\begin{lemma}\label{Outer Rectangle}
    There is a $\lambda_4\in(0,\lambda_3)$ such that for each $\lambda\in[0,\lambda_4]$, $m\in\mathbb{N}$, and $i=1,\dots, 4$, there is a dynamical rectangle $R_m(\lambda, Q_i)$ centered at $Q_i$ with respect to $\tilde K(Q_i)$ such that
    \begin{itemize}
        \item[1.)] $\partial^sR_m(\lambda, Q_i)\subset W^s_{\lambda}(Q_0,m)$ and in particular, $\partial^s_+R_m(\lambda, Q_i)$ connects opposite boundaries of $\tilde K_{+}(Q_i)$ and $\partial^s_-R_m(\lambda, Q_i)$ connects opposite boundaries of $\tilde K_{-}(Q_i)$\emb{.}
        \item[2.)] For any $P\in \partial^sR_m(\lambda, Q_i)$, we have $|P-Q_i|<\frac{C^*}{2}$\emb{.}
        \item[3.)] If $P\in \partial^sR_m(\lambda, Q_i)$, then
       $$\min\left\{\text{dist}\left(P, \bigcup_{i=1}^4\{Q_i\}\right), \text{dist}\left(f_{\lambda, n, m}(P), \bigcup_{i=1}^4\{Q_i\}\right)\right\}\geq \frac{C^{**}}{2}$$
for any $n\geq 0$\emb{.}

    \end{itemize}
   
\end{lemma}
\begin{proof}
We break the proof into steps:\\
\\
\noindent\textbf{Step 1:} We first construct the stable sides of the rectangles for $\lambda=0.$ 

We know that $W^s_0(Q_0,m)$ is a line through $Q_0$ of slope $-\frac{1}{[a_m,a_{m+1},\dots]}$ \emb{(see Lemma \ref{slope of linear stable manifolds})} and we know from Lemma \ref{lengths of curves} that 
$$\text{length}\left(f^{-1}_{0,\tilde n, m}\left(W^s_{0, \text{loc}}(f_{0, \tilde n, m}(Q_0), m+\tilde n+1) \right)\right)\geq l^*,$$
where $\tilde n$ was chosen right after Lemma \ref{upper constant and length}.

From Lemma \ref{upper constant and length}, we know that $f^{-1}_{0,\tilde n, m}\left(W^s_{0, \text{loc}}(f_{0, \tilde n, m}(Q_0), m+\tilde n+1) \right)$ has two segments, denoted by $\gamma^{+}_{i,m}$ and $\gamma^{-}_{i.m}$, that connect the opposite lines that make up $\partial ^u\mathbb{B}_i$ and are centered at $Q_i$ with respect to $\tilde K(Q_i)$, where $\gamma^{\pm}_{m,i}\cap \tilde K_{\pm}(Q_i)\neq \emptyset$. In addition, these segments have distance from $Q_i$ less than $\cos(\theta)\frac{C^{*}}{4}$, where $\theta$ is the angle between boundary lines of the cone $\tilde K(Q_i)$. We verify these segments satisfy certain estimates:
 
Let $\bullet\in\{+,-\}$, $P\in\gamma^{\bullet}_{i,m}\cap \tilde K(Q_i)$, and choose $P'\in \gamma^{\bullet}_{i,m}\cap \tilde K(Q_i)$ to be the point closest to $Q_i$. Then, the line segment connecting $P'$ and $Q_i$ is normal to $\gamma_{i,m}^{\bullet}$, and hence contained in $\tilde K^u(Q_i)$. Therefore,

\begin{equation}\label{linear base rectangle upper estimate}
    |P-Q_i|\leq  \cos(\theta)^{-1}|P'-Q_i|<\frac{C^{*}}{4}.
\end{equation}
Suppose that $P\in \gamma^{+}_{i,m}\cup \gamma^{-}_{i,m}$ and let $n\geq 0$. From Lemma \ref{lower constant and length}, we know $|P-Q_i|\geq C^{**}$ since 
$$\text{length}\left(f^{-1}_{0,\tilde n, m}\left(W^s_{0, \text{loc}}(f_{0, \tilde n, m}(Q_0), m+\tilde n+1) \right)\right)\leq \frac{(\mu_2+\delta)^{\tilde n+1}}{10}<l^{**},$$
from Lemma \ref{lengths of curves}.

Notice that $f_{0,\tilde n,m}(P)\in W^s_{0, \text{loc}}(Q_0, m+\tilde n+1)$, and so if $n\geq \tilde n$, then $f_{\lambda, n, m}(P)\in B_{\frac{1}{10}}(Q_0)$ so that its distance from $\bigcup_{i=1}^4 \{Q_i\}$ is greater than $\frac{3}{20}\emb{=}\frac{1}{4}-\frac{1}{10}$.

  If $n< \tilde n$, then  $f_{0,  n, m}(P)\in f^{-1}_{0, \tilde n-n,m+n +1} W^s_{0,\text{loc}}(f_{0,\tilde n, m}(Q_0), \tilde n+m+1)$, where $f^{-1}_{0, \tilde n-n,m+n +1} W^s_{0,\text{loc}}(f_{0,\tilde n, m}(Q_0), \tilde n+m+1)$ is a line segment contained in $W^s_{0}(f_{0,n, m}(Q_0), m+n+1)$ that has slope $-\frac{1}{[a_{m+n+1}, a_{m+n+2},\dots]}$, contains $f_{0,n,m}(Q_0)\in\text{Orb}_{0,m}(Q_0)$, and has length no greater than $\frac{(\mu_2+\delta)^{(\tilde n-n)+1}}{10}<l^{**}$. It follows from Lemma \ref{lower constant and length} that 

\begin{equation}\label{linear rectangle lower estimate}
   \min\left\{\text{dist}\left(P, \bigcup_{i=1}^4\{Q_i\}\right), \text{dist}\left(f_{0, n, m}(P), \bigcup_{i=1}^4\{Q_i\}\right)\right\}\geq C^{**}
\end{equation}
for all $n\geq 0$.\\
\\
\noindent\textbf{Step 2:} We now construct the stable sides of the rectangles for a small range of $\lambda$. 

For $\lambda=0$ and $\bullet\in\{+,-\}$, we let $\psi_{0,i,m}^{\bullet}$ be the segment of $\gamma^{\bullet}_{i,m}$ that connects opposite boundaries of $\tilde K_{\bullet}(Q_i)$.

From Corollary \ref{perturbations of stable manifolds} \emb{and Lemma \ref{common orbit}},  the quantity
\begin{equation}\label{closeness of finite pieces of stable manifolds}
    \sup\limits_{\substack{k\in\mathbb{N} \\ 0\leq n \leq \tilde n \\ m\in\mathbb{N}}}\emb{\text{dist}_H}\left( f^{-1}_{\lambda, n, m} W^s_{\lambda, \text{loc}}(Q, k),  f^{-1}_{0, n, m} W^s_{0, \text{loc}}(Q, k)\right)
\end{equation}

can be made arbitrarily small for a sufficiently small range of $\lambda$ for all $Q\in\text{Orb}_{\lambda}(Q_0)$.

Then in particular, if $f^{-1}_{\lambda,\tilde n, m}\left(W^s_{\lambda, \text{loc}}(Q_0, m+\tilde n+1) \right)$ and $f^{-1}_{0,\tilde n, m}\left(W^s_{0, \text{loc}}(Q_0, m+\tilde n+1) \right)$ are close enough, there will be a curve $\gamma^{\bullet}_{\lambda, i,m}\subset f^{-1}_{\lambda,\tilde n, m}\left(W^s_{\lambda, \text{loc}}(Q_0, m+\tilde n+1) \right)$ that exhibits a sub-curve $\psi^{\bullet}_{\lambda, i,m}\subset \gamma^{\bullet}_{\lambda, i,m}$ such that $\psi^{\bullet}_{\lambda, i,m}$ connects opposite boundaries of $\tilde K_{\bullet}(Q_i)$. Using estimates (\ref{linear base rectangle upper estimate}) and (\ref{linear rectangle lower estimate}), we may find a $\lambda_4\in(0,\lambda_3)$ such that for all $\lambda\in[0,\lambda_4]$, we have

\begin{itemize}
    \item For any $P\in \psi_{\lambda,i,m}^+\cup \psi_{\lambda,i,m}^{-} $, we have $|P-Q_i|<\frac{C^*}{2}$
    \item For any $P\in \psi_{\lambda,i,m}^+\cup \psi_{\lambda,i,m}^{-} $, we have 
         $$\min\left\{\text{dist}\left(P, \bigcup_{i=1}^4\{Q_i\}\right), \text{dist}\left(f_{\lambda, n, m}(P), \bigcup_{i=1}^4\{Q_i\}\right)\right\}\geq \frac{C^{**}}{2}$$
         for all $n\geq 0$
\end{itemize}

\noindent\textbf{Step 3:} We now construct the rectangles. 

For each $\lambda\in[0,\lambda_4]$ we choose $R_m(\lambda, Q_i)$ to be the rectangle containing $Q_i$ such that $\partial ^s_{\bullet} R_m(\lambda,Q_i)=\psi^{\bullet}_{i,m}$ for $\bullet\in\{+,-\}$, and $\partial^{ns}R_m(\lambda,Q_i)$ is the pair of line segments that connect appropriate endpoints of these stable boundaries.
\end{proof}

\begin{figure}
    \begin{center}
        \begin{tikzpicture}[scale=1.2]
\begin{axis}[xtick=\empty, ytick=\empty, axis line style={draw=none}, tick style={draw=none}]

\addplot[name path=C, domain={-3.5:3.5}]{0.95*x};   
\addplot[name path=D, domain={-3.5:3.5}]{-0.95*x};

 \addplot[fill=black, fill opacity=0.1] fill between [of=C and D];

\draw[blue, very thick] (-1.8,0.95*-1.8)  .. controls (-2, .5) ..  (-2,-0.95*-2);
\draw[blue, very thick] (2.1, 0.95*2.1) -- (-2,-0.95*-2) node[midway, above]{\scriptsize $R_{m}(\lambda,Q_i)$};
\draw[blue, very thick] (-1.8,0.95*-1.8) -- (2,-0.95*2);
\draw[blue, very thick] (2.1, 0.95*2.1).. controls (1.7, .5) .. (2,-0.95*2) ;   

\filldraw[black] (0,0) circle (2pt);

\draw[red, dashed] (-2,.5) -- (0,0) node[midway, below]{\tiny $\leq \frac{C^*}{2}$};
\draw[red, dashed] (1.7, .5) -- (0,0) node[midway, below]{\tiny $\geq \frac{C^{**}}{2}$};
\coordinate[label=below right: \tiny $\tilde K(Q_i)$] (A) at (-2.9,-2.9*0.95);

\end{axis}

\end{tikzpicture}

    \caption{Rectangle from Lemma \ref{Outer Rectangle}.}
    \label{Outer Rectangle Figure}

    \end{center}
\end{figure}
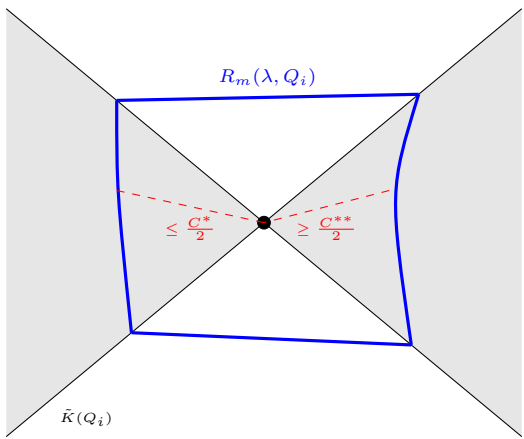

\subsubsection{Rectangles of Controlled Size}
We will now use this base set of rectangles to construct dynamical rectangles of small controlled size. We will need rectangles of smaller scale to ensure that the \emb{sub-}Cantor set of $\mathcal C(\lambda, \emb{\tilde O_{r(\rho)}})$, that we construct later on, will be \emb{thick and hence have Hausdorff dimension larger than $\frac{\log2}{\log\left(2+\frac{1}{M}\right)}$}.

The construction will come from taking the stable boundaries of rectangles from the collection $\left\{R_{\lambda}(m, Q_i)\right\}$, and iterating them under the maps $f^{-1}_{\lambda, n, m}$, which will give two new curves tangent to $\mathcal K^s$ that have segments which are close together and centered at some $Q_i$, and these curves will serve as the stable boundaries for a new rectangle. 

\begin{lemma}\label{inner rectangle construction}
    For each $\lambda \in[0,\lambda_4]$, $m\in\mathbb{N}$, $i=1,\dots,4$, and $N\in \mathbb N\cup\{ 0\}$, there is a dynamical rectangle $R_m'(\lambda, Q_i, N)$ centered at $Q_i$ with respect to $\tilde K(Q_i)$ satisfying:
    \begin{itemize}
        \item[1)] The stable boundaries $\partial^s R_m'(\lambda, Q_i, N)$ satisfy the following: 
        
        For each $\bullet\in\{+,-\}$, we have that $\partial^s_{\bullet} R_m'(\lambda, Q_i,N)$ connects opposite boundaries of $\tilde K_\bullet(Q_i)$. Moreover,
        
        $$\partial^s_{\bullet} R_m'(\lambda, Q_i, N)\subset f^{-1}_{\lambda, N,m}\left[\partial ^s_{\bullet} R_{m+N+1}\left(\lambda, f_{\lambda, N, m}(Q_i) \right)\right]$$
          and in particular
       $$\partial^s_{\bullet} R_m'(\lambda, Q_i, N)\subset \hat T_{\lambda, a_m}^{-1}\left(\partial^s_{\bullet} R'_{m+1}\left(\lambda, \hat T_{\lambda, a_m}(Q_i), N-1\right)\right).$$
      
        \item[2)] If $\kappa$ is a curve contained in $\mathbb B_i$ and tangent to $\mathcal K^u$ that connects $\partial_{+}^s R'_m(\lambda, Q_i, N)$ and $\partial_{-}^sR'_m(\lambda, Q_i, N)$, then $$\text{length}(\kappa) \leq C_1(\mu_1-\delta)^{-N} C^{*}\emb{.}$$
        Similarly, if $\kappa$ instead connects $Q_i$ and $\partial^s R'(\lambda, Q_i, N)$, for some $\bullet\in\{+,-\}$, then 
        $$\text{length}(\kappa) \leq C_1(\mu_1-\delta)^{-N} \frac{C^{*}}{2}\emb{.}$$
        \item[3)] We have 
        $$\text{dist}\left(\partial_{\bullet}^s R'_m(\lambda, Q_i, N) , Q_i\right) \geq  (\mu_2+\delta)^{-N-1}\frac{C^{**}}{2},$$ for each $\bullet\in \{+,-\}\emb{.}$

        \item[4)] The non-stable boundaries $\partial^{ns}R'_m(\lambda, Q_i, N)$ of $R'_m(\lambda, Q_i, N)$ are line segments with endpoints in $\partial \tilde K(Q_i)$. That is, $R'_m(\lambda, Q_i, N)$ has vertices belonging to $\partial \tilde K(Q_i)$, and does not intersect $\emb{\partial}\tilde K(Q_i)$ otherwise.

    \end{itemize}
\end{lemma}

\begin{proof}
Let $\lambda\in[0,\lambda_4]$.
For $N=0$, from Lemma \ref{Outer Rectangle}, we set $R'_m(\lambda, 0, Q_i)=R_m(\lambda, Q_i)$, and these rectangles satisfy the desired properties.

We will proceed inductively. Let $N\geq 1$, and suppose that we have constructed the family of rectangles $\left\{ R'_m(\lambda, Q_i, k) : m\in\mathbb{N}, i=1,\dots, 4 \right\}$ for each $k<N$ and that they satisfy the desired properties.

Let $m\in \mathbb{N}$ and $i\in\{1,\dots, 4\}$. Consider the rectangle $R'_{m+1}\left( \lambda, Q_j, N-1\right)$, where $Q_j=f_{\lambda, 0,m}(Q_i)$, which is centered at $Q_j$ with respect to $\tilde K(Q_j)$. 

Notice from (\ref{Orientation for toral maps}) and \emb{Lemma} \ref{linear overflow}, provided $\lambda$ is small, the image of the cone $\tilde K(Q_i)$, under the map $f_{\lambda, 0, m}$, overflows the box $\mathbb B_j$ (see Figure \ref{inner rectangle construction picture}). Without loss of generality, let us assume that this holds for all $\lambda\in[0,\lambda_4].$

Let us denote the component of $f_{\lambda, 0, m}\left(\tilde K(Q_i)\right)\cap \mathbb B_j$ that contains $Q_j$ by $\tilde K^*$. Then, $\tilde K^*$ is contained in $\tilde K(Q_j)$ and for each $\bullet\in\{+,-\}$, the boundary curves of $\tilde K^*\cap \tilde K_{\bullet}(Q_j)$ are tangent to $\mathcal K$ and connect $Q_j$ and a boundary line of $\partial ^s\mathbb B_j$. It follows from Remark \ref{lifted curve orientation} that $f^{-1}_{\lambda, 0, m}\left(\partial^s_{\bullet}R'_{m+1}(\lambda, Q_j, N-1)\cap \tilde K^*(Q_j)\right)$ is a curve contained in $\mathbb B_i$ that is tangent to $\mathcal K^s$ and connects the boundary lines of $\tilde K_{\bullet}(Q_i)$. 

We set $R'_m(\lambda, Q_i, N)$ to be the rectangle containing $Q_i$ such that $\partial^s_\bullet R'_m(\lambda, Q_i, N):=f^{-1}_{\lambda, 0, m}\left(\partial^s_{\bullet}R'_{m+1}(\lambda, Q_j, N-1)\cap \tilde K^*(Q_j)\right)$ for each $\bullet\in\{+,-\}$, and whose non-stable boundaries are straight line segments. 

By construction and the induction hypothesis, we have
\begin{align*}
    &\partial^s_{\bullet}R'_{m}(\lambda, Q_i, N)\subset \hat T_{\lambda, a_{m}}^{-1}\left(  \partial^s_{\bullet} R'_{m+1}(\lambda, Q_j, N-1)\right)\\
    & \subset \hat T_{\lambda, a_{m}}^{-1}\left(f^{-1}_{\lambda, N-1, m+1}\left[ \partial^s_{\bullet}R_{(m+1)+(N-1)+1}(\lambda, Q_k) \right] \right)\\
    &=f^{-1}_{\lambda, N, m}\partial^s_{\bullet}R_{m+N+1}\left(\lambda, f_{\lambda, N, m}(Q_i) \right).
\end{align*}
This establishes (1) and (4), and we now check that this rectangle satisfies (2) and (3).

Let $\bullet\in\{+,-\}$. Suppose that $\kappa$ be a curve tangent to $\mathcal K^u$ that connects $Q_i$ and $\partial^s_{\bullet}R'_m(\lambda, Q_i, N)$. Then, $f_{\lambda, 0, m}(\kappa)$ is a curve tangent to $\mathcal K^u$ that connects $Q_j$ and $\partial ^sR'_{m+1}(\lambda, Q_j, N-1)$. By (2) of Lemma \ref{lengths of curves}, the induction hypothesis (in particular property (2)), and \emb{Lemma} \ref{curve length compared to distance}, we have 
$$\text{length}(\kappa)\leq (\mu_1-\delta)^{-1}\text{length}\left(f_{\lambda, 0,m}(\kappa)\right)\leq (\mu_1-\delta)^{-1}C_1(\mu_1-\delta)^{-(N-1)}\frac{C^*}{2}$$
so that second part of (2) holds, and the first part is verified in an almost identical way. 

In a similar manner, again from (2) of Lemma \ref{lengths of curves} and the induction hypothesis (in particular property (3)), we also have  
\begin{align*}
    \text{length}(\kappa) &\geq (\mu_2+\delta)^{-1}\text{length}\left(f_{\lambda, 0.m}(\kappa)\right)\geq (\mu_2+\delta)^{-1}\text{dist}\left(\partial^s_\bullet R'_{m+1}(\lambda, Q_j, N-1), Q_j\right)\\
    &\geq (\mu_2+\delta)^{-1}\frac{C^{**}}{2}(\mu_2+\delta)^{-N-2}=(\mu_2+\delta)^{-N-1}\frac{C^{**}}{2},
\end{align*}
and since $\partial^s_{\bullet}R'_m(\lambda, Q_i, N-1)$ is tangent to $\mathcal K^s$ and $\kappa$ was arbitrary and tangent to $\mathcal K^u$, we conclude that (3) holds.
\end{proof}

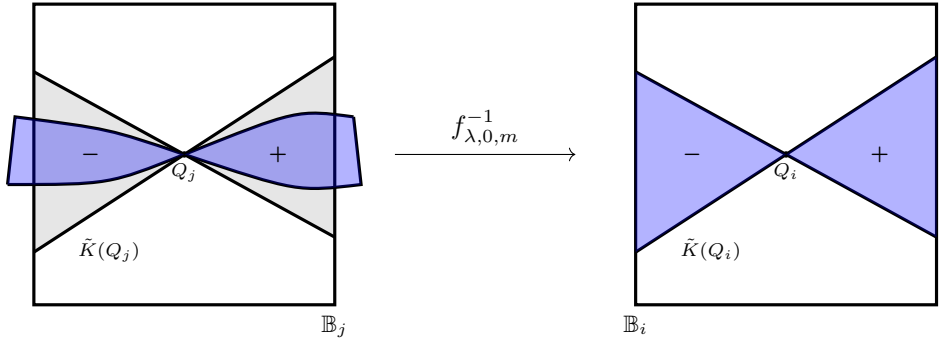
\begin{figure}[h]
   \begin{center}
    
   \begin{tikzpicture}

\filldraw[black] (-4, 0) circle (1pt) node[below]{\scriptsize $Q_j$};
\draw[black, very thick] (-6,2) rectangle (-2,-2) node[anchor=north]{\small $\mathbb B_j$};

\filldraw[black] (4, 0) circle (1pt) node[below]{\scriptsize $Q_i$};
\draw[black, very thick] (6,2) rectangle (2,-2) node[anchor=north]{\small $\mathbb B_i$};

\draw[black, very thick] (6, .65*2)-- (2,.65*-2);
\draw[black, very thick] (6, -.55*2)-- (2,-.55*-2);

\fill[blue, opacity=0.3] (6, .65*2) -- (4,0) -- (6, -.55*2)-- (6, .65*2);

\fill[blue, opacity=0.3] (2,.65*-2) -- (4,0) -- (2,-.55*-2) -- (2,.65*-2);

\draw[black, very thick] (-6, .55*2)-- (-2,.55*-2);
\draw[black, very thick] (-6, .65*-2)-- (-2,.65*2);

\fill[black, opacity=0.1] (-6, .55*2) -- (-4,0) -- (-6, -.65*2) -- (-6, .55*2);

\fill[black, opacity=0.1] (-2, .65*2) -- (-4,0) -- (-2, -.55*2) -- (-2, .65*2);

\draw[black, very thick] (-4,0) .. controls (-5,.17*2) .. (-6.25, .25*2) -- (-6.35, -.2*2);

\draw[black, very thick] (-4,0) .. controls (-5,-.2*2) ..(-6.35, -.2*2);

\fill[blue, opacity=0.3] (-4,0) .. controls (-5,.17*2) .. (-6.25, .25*2) -- (-6.35, -.2*2) .. controls (-5,-.2*2) .. (-4,0);

\draw[black, very thick] (-4,0) .. controls (-2.5,-.25*2) ..(-1.65, -.2*2) -- (-1.75, .25*2);

\draw[black, very thick] (-4,0) .. controls (-2.5,.3*2) ..(-1.75, .25*2);

\fill[blue, opacity=0.3] (-4,0) .. controls (-2.5,-.25*2) ..(-1.65, -.2*2) -- (-1.75, .25*2) .. controls (-2.5,.3*2) .. (-4,0);

\draw[->]  (-1.2,0) -- (1.2,0) node[midway, above]{$f^{-1}_{\lambda, 0, m}$};

\coordinate[label=left: \scriptsize $\boldsymbol -$] (C) at (-5,0);

\coordinate[label=right: \scriptsize $\boldsymbol +$] (C) at (-3,0);

\coordinate[label=left: \scriptsize $\boldsymbol -$] (C) at (3,0);

\coordinate[label=right: \scriptsize $\boldsymbol +$] (C) at (5,0);

\coordinate[label=below: \scriptsize $\tilde K(Q_j)$] (C) at (-5,-1);

\coordinate[label=below: \scriptsize $\tilde K(Q_i)$] (C) at (3,-1);
\end{tikzpicture}

    \caption{Illustration of the image of $\tilde K(Q_i)$ under the map $f_{\lambda, 0, m}$.}
    \label{inner rectangle construction picture}

    \end{center}
\end{figure}
We note that combining (2) and (4) of this Lemma implies that 
\begin{equation}\label{diameter of rectangles}
    \lim\limits_{N\rightarrow \infty}\text{diam}\left(R_{m}(\lambda, Q_i, N)\right)=0,
\end{equation}
where this limit is uniform in $\lambda\in[0,\lambda_4]$, $m\in\mathbb N$, and $i=1,\dots, 4$.

\subsection{Gap Estimates and Thickness}\label{thickness estimate section}
In order to establish (\ref{Hausdorff dimension of torus Cantor}), we will show that for each $\lambda\in[0,\lambda_4]$, there is a Cantor set $\mathcal C\left(\lambda, N^*\right)$ contained in $ \mathcal C\left(\lambda, \emb{\tilde O_{r(\rho)}}\right)$ that has \emb{Hausdorff dimension larger than $\frac{\log 2}{\log\left(2+\frac{1}{M}\right)}$.}

\emb{\subsubsection{Thickness of Cantor Sets}\label{Thickness of Cantor Sets Section}
A fundamental tool we will use to estimate \emb{the Hausdorff dimension of this set} is the notion of thickness for Cantor sets. Given a Cantor set $\mathcal C\subset \mathbb R$ with convex hull $I$, we say that $U\subset I\setminus \mathcal C$ is a \textit{gap} of $\mathcal C$ if it is a connected component of $I\setminus \mathcal C$. Any enumeration $\mathcal U=\{U_n\}$ of the gaps of $\mathcal C$ is called a \textit{presentation} of $\mathcal C$. Given a boundary point $u\in \mathcal C$ of some gap $U_k$, we call the connected component $B$ contained in $I\setminus\left(U_1\cup\dots\cup U_k\right) $ that contains $u$ a \textit{bridge} of $U_k$. We define 
$$\tau(\mathcal C, \mathcal U, u):=\frac{|B|}{|U_k|},$$
and set the \textit{thickness} of $\mathcal C$ to be 
$$\tau(\mathcal C):=\sup\limits_{\mathcal U}\inf\limits_{u}\tau(\mathcal C, \mathcal U, u).$$
It is related to the Hausdorff dimension of $\mathcal C$ by the inequality
\begin{equation}\label{Thickness and Hausdorff dimension}
    \text{dim}_{H}(\mathcal C)\geq \frac{\log2 }{\log\left(2+\frac{1}{\tau(\mathcal C)}\right)}.
\end{equation}
(see Chapter 4 of \cite{pt2}). It is clear that these notions carry over to a Cantor set that belongs to a $C^1$ curve $\gamma\subset \mathbb T^2$ where the lengths of gaps and bridges are measured using length relative to the curve $\gamma$.}

\subsubsection{\emb{Ordering of $\hat L_{\lambda}$}}

For each curve $\hat L_\lambda$, as it is close to a segment of $\hat L_0=\left[0,\frac{1}{2}\right]\times\{0\}$ and tangent to $\mathcal K^u$, for our range of $\lambda$, we assign a natural ordering to its elements as follows:
Let $P^*$ be the endpoint of $\hat L_\lambda$ that is closest to $(0,0)$ and $P^{**}$ its other endpoint. We see that the endpoints will be close $(0,0)$ and $\left(\frac{1}{2},0\right)$, respectively, for small $\lambda$. We say that $P< P'$ if 
$$\text{dist}_{\hat L_\lambda}\left(P,P^*\right)< \text{dist}_{\hat L_\lambda}\left(P',P^*\right).$$
With this in mind, we have the following lemma.
\begin{lemma}\label{orientation of endpoints}
    Let $I\subset \hat L_\lambda$ be a segment with endpoints $P<P'$ such that $f_{\lambda, n, m}(P)\in \tilde K_{+}(Q_i)$ and $f_{\lambda, n, m}\left(P'\right)\in \tilde K_{-}(Q_j)$ for some $i, j\in\{1,\dots, 4\}$. Then, there must be a $P_0\in \hat L_\lambda $ between $P$ and $P'$ such that $f_{\lambda, n, m}(P_0)\not\in \bigcup_{i=1}^4\mathbb B_i$.
\end{lemma}
\begin{proof}
If $i\neq j$, then this is clear. For $i=j$, if we parametrize $I$ as a curve with positive orientation, then this follows from Remark \ref{lifted curve orientation}.
\end{proof}
\subsubsection{Gap Constructions}
Recall from (\ref{Rectangle not too big}) we know $R'_{k}(\lambda, Q_i, N^*)\subset B_{\frac{C^{**}}{4}}(Q_i)$ for all $i=1,\dots, 4$ and $k\in\mathbb{N}$.

We will construct a Cantor set $\mathcal C\left(\lambda, N^*\right)$ that is contained in $\mathcal C\left({\lambda}, \emb{\tilde O_{r(\rho)}}\right)$ and show that it has thickness bigger than $M$: 

Consider the line $\hat L_\lambda $ and recall that it is tangent to $\mathcal K^u$ for any $\lambda\in[0,\lambda_4]$. As $\hat L_\lambda$ tends to $\hat L_0\setminus O_{r_0}$ as $\lambda\rightarrow 0$, and $B_{r_0}(0,0)\subset R_1(\lambda, (0,0))$ and $B_{r_0}\left(\frac{1}{2}, 0\right)\subset R_1\left(\lambda, \left(\frac{1}{2}, 0\right)\right)$, we can work in a small enough range of $\lambda$ so that $\hat L_{\lambda}$ has its least and greatest endpoints belonging to $ R_{1}(\lambda, (0,0))$ and $R_{1}\left(\lambda, \left(\frac{1}{2},0\right)\right)$ respectively. If $\lambda$ is sufficiently small, then we can assume that $\hat L_{\lambda}$ intersects the stable sides of these rectangles since this is the case for $\hat L_0$. Without loss of generality, let us assume this holds for all $\lambda\in[0,\lambda_4]$.

Then, there exist $P_0^{+},P_0^-\in \hat L_{\lambda}$ such that $P_0^+<P_0^-$ and $P_0^+\in \partial^s_+ R_{1}(\lambda, (0,0))$ and $P_0^-\in \partial^s_-R_{1}\left(\lambda, \left(\frac{1}{2},0\right)\right)$.

We notice that 
\begin{equation}\label{initial endpoint orbits}
    \text{dist}\left( f_{\lambda, k, 0}\left(P_0^\bullet\right),  \bigcup_{i=1}^4 \{Q_i\}\right)\geq \frac{C^{**}}{2}
\end{equation}
for all $k\geq 0$ and each $\bullet\in\{+,-\}$.

Suppose that $I_0$ is the segment from $P_0^{+}$ to $P_0^-$ contained in $\hat L_\lambda$, and let $k_0$ be the smallest value such that $f_{\lambda, k_0, 0}\left(I_0 \right)$ intersect $\bigcup_{i=1}^4O_{r_0}(Q_i)$. Now, we consider all of the components of $$f_{\lambda, k_0, 0}(I_0 )\cap \left(\bigcup _{i=1}^4 R'_{k_0+1}\left(\lambda, Q_{i}, N^*\right)\right).$$ The gaps of order $k_0$ will come from taking the image of such components that intersects $\bigcup_{i=1}^4 O_{r_0}(Q_i)$ non-trivially, under the map $f^{-1}_{\lambda, k_0, 0}$. 

Let $J$ be such a component, and suppose it is contained in $R'_{k_0+1}(\lambda, Q_{j}, N^*)$ for some $j=1,\dots, 4$, and $J\cap O_{r_0}(Q_j)\neq \emptyset$. Then, $J$ is the image of some segment $I_0'\subset I_0$ under the map $f_{\lambda, k_0, 0}$. We will say that $\emb{\text{int} \left(I'_0\right)}$ is a gap of order $k_0$, where the interior is taken with respect to the subspace topology on $\hat L_\lambda$. 

Since $R'_{k_0+1}\left(\lambda, Q_{j}, N^*\right)$ is centered at $Q_j$ with respect to $\tilde K(Q_j)$ and \newline $f_{\lambda, k_0, 0}\left(I_0'\right)$ is a curve tangent to $\mathcal K^u$, by \emb{Proposition} \ref{outer cone centering gives inner ball} and our discussion after Lemma \ref{inner rectangle construction}, as $f_{\lambda, k_0, 0}\left(I_0'\right)\cap O_{r_0}(Q_j)\neq \emptyset$, we must have that $f_{\lambda, k_0, 0}\left(I_0'\right)$ is fully contained in $R'_{k_0+1}(\lambda, Q_j, N^*)$ and only intersects $\partial R'_{k_0+1}(\lambda, Q_j, N^*)$ non-trivially in the event that either of its endpoints belong to $\partial^sR'_{k_0+1}(\lambda, Q_j, N^*)$. However, from (\ref{initial endpoint orbits}), and since we chose $N^*$ so large that $R'_{k_0+1}(\lambda, Q_j, N^*)\subset B_{\frac{C^{**}}{4}}$  (see (\ref{Rectangle not too big})), it must be the case that $f_{\lambda, k_0, 0}\left(I_0'\right)$ indeed connects the stable boundaries of $R'_{k_0+1}(\lambda, Q_j, N^*)$. In particular, by Lemma \ref{orientation of endpoints}, if $I_0'$ has endpoints $P_1$ and $P_2$ with $P_1<P_2$, then we have $f_{\lambda, k_0, 0}(P_1)\in \partial^s_{-}R'_{k_0+1}(\lambda, Q_j, N^*)$ and $f_{\lambda, k_0, 0}(P_2)\in \partial^s_{+}R'_{k_0+1}(\lambda, Q_j, N^*)$.

Suppose that we have constructed all of the gaps up to order $k\geq k_0$, and that they are all disjoint segments of $\hat L_{\lambda}$ of the form $\emb{\text{int}\left(I_m\right)}$ where $I_m\subset I_0$ for some $m\leq k$ such that:

\begin{itemize}
    \item $f_{\lambda, k,0}\left(I_m\right)$ is a curve contained in $R'_{m+1}( \lambda, Q_i, N^*)$, for some $i=1,\dots, 4$
    \item $I_m$ has endpoints $P_m^-<P_m^+$ so that
    $f_{\lambda, m, 0}\left(P_m^- \right)\in \partial^s_-R'_{m+1}(\lambda, Q_i, N^*) $ and $f_{\lambda, m, 0}\left(P_m^+ \right)\in \partial^s_+R'_{m+1}(\lambda, Q_i, N^*) $
    \item $f_{\lambda, m, 0}\left(I_m\right)\cap O_{r_0}(Q_i)\neq\emptyset $
\end{itemize}

Denote this collection of gaps by $\mathcal U_k$. Let $I$ be a component of $I_0\setminus \mathcal U_k$.
Now, consider the components of  $$f_{\lambda, k+1, 0}(I)\cap \left(\bigcup _{i=1}^4 R'_{k+2}(\lambda, Q_i, N^*)\right).$$ The gaps of order $k+1$ will come from taking the image of such components that intersects $\bigcup_{i=1}^4 O_{r_0}(Q_i)$ non-trivially, under the map $f^{-1}_{\lambda, k+1, 0}$. A similar application of \emb{Proposition} \ref{outer cone centering gives inner ball} and Lemma \ref{orientation of endpoints} as in the base case shows that any gap of order $k+1$ will be of the form above with index $m$ replaced with $k+1$.

     We have now constructed all orders of gaps of $\mathcal C\left(\lambda, N^*  \right)$ and shown they have satisfied a certain form. From (\ref{fixed inner ball radius}) and definition of $r_0$, we know 
 
 $$B_{r(\rho)}(Q_i)\subset B_{r_0}(Q_i)\subset R'_{k}(\lambda, Q_i, N^*)$$ for all $k\geq 0$, $i=1,\dots, 4$ and $\lambda\in[0,\lambda_4]$. Thus, it follows that 
\begin{equation}\label{Cantor set containment}
    C\left(\lambda, N^*\right)\subset C\left(\lambda, \emb{\tilde O_{r(\rho)}}\right).
\end{equation}

 \subsection{\emb{Concluding the Proof of Theorem \ref{intermediate theorem}}}\label{concluding the proof of intermediate theorem}
\emb{With the construction of the set $\mathcal C\left(\lambda, N^*\right)$, we are now ready to complete the proof of Theorem \ref{intermediate theorem}.}
 \begin{proof}[\emb{Proof of Theorem \ref{intermediate theorem}.}]
     
  Let $\lambda_0=\min\{\lambda_1, \lambda_2, \lambda_3, \lambda_4\}$ and let $\lambda\in[0,\lambda_0]$. We will show that $C\left(\lambda, N^*\right)$ has thickness larger than $M$.

With respect to the Cantor set $\mathcal C\left(\lambda, N^*\right)$ and the ordering of gaps defined above, let $U\subset I_0$ be a gap of order $k$ with endpoints $P<P'$, and $B$ a bridge of it. Without loss of generality, suppose that $P$ is an endpoint of $B$, and denote its other endpoint by $P''$. We know that $f_{\lambda, k, 0}\left( P\right)\in \partial ^s_{-}R'_{k+1}(\lambda, Q_i, N^*)$ and $f_{\lambda, k, 0}\left( P' \right)\in \partial ^s_{+}R'_{k+1}(\lambda, Q_i, N^*)$, and there must be some $Q'$ between $P$ and $P'$ such that $f_{\lambda,k,0}\left(Q'\right)\in O_{r_0}(Q_{i})$.

\begin{lemma}\label{intermediate point}
    There is a point $Q\in B$ between $P$ and $P''$  such that 
    $$\text{dist}\left( f_{\lambda, k, 0}(Q),  \bigcup_{i=1}^4 \{Q_i\}\right)\geq \frac{C^{**}}{2},$$
    and the segment $J\subset f_{\lambda,k,0}(B\cup U)$ connecting $f_{\lambda,0,k}(Q)$ and $f_{\lambda,  k,0}(P)$ is contained in $\mathbb B_i$.
\end{lemma}

\begin{proof}[\emb{Proof of Lemma \ref{intermediate point}.}]
     By the construction of our gaps, the endpoint $P''$ of $B$ satisfies $f_{\lambda, k', 0}\left(P''\right)\in \partial^sR'_{k'+1}(\lambda, Q_j, N^*)$ for some $j=1,\dots, 4$ and $k'\leq k$.

    From (1) of Lemma \ref{inner rectangle construction}, either $f_{\lambda, k,0}\left(P''\right)\in \partial^s_{+}R_{k+1}(\lambda, Q_j, N')$ for some $N'\leq N^*$ and $j=1,\dots, 4$, or there was some $k''\in\left[k',k\right)$ so that $f_{\lambda, k'',0}\left(P''\right)\in \partial^s_{+}R_{k''+1}(\lambda, Q_l, 0)$ for some $l=1,\dots, 4$. 

    In the former case, Lemma \ref{orientation of endpoints} implies that there is a point $Q^*\in B$ such that $f_{\lambda, k, 0}(Q^*)\not\in \mathbb B_i$. In this case, we notice that the segment of $f_{\lambda, k, 0}(B\cup U)$ connecting $f_{\lambda, k, 0}(Q^*)$ and $f_{\lambda, k , 0}(P')$ intersects $O_{r_0}$, and from our choice of neighborhood $O_{r_0}$ (see Section \ref{Translation to a Toral Map Problem}), this segment must intersect $ \partial^s_{-}R_{k+1}(\lambda, Q_i)$, in which case, we choose $Q$ so that $f_{\lambda, k, 0}(Q)$ is this intersection point, and we recall (3) of Lemma \ref{Outer Rectangle}.
    
    In the latter case, by (3) of Lemma \ref{Outer Rectangle}, $f_{\lambda, k,0}(P'')$ has distance from $Q_i$ greater than $\frac{C^{**}}{2}$. If the segment of $f_{\lambda, k, 0}(B\cup U)$ connecting $f_{\lambda, k, 0}(P'')$ and $f_{\lambda, k, 0}(P')$ is not contained in $\mathbb B_i$, then by the same reasoning as in the former case, there is some $Q\in B$ such that $f_{\lambda, k, 0}(Q)\in \partial^s_{-}R_{k+1}(\lambda, Q_i)$. 
\end{proof}

This \emb{lemma} gives the following estimate:
\begin{align*}
&\text{dist}_{f_{\lambda, k,0}\left(\hat L_\lambda \right)}\left(f_{\lambda,k,0}(P), f_{\lambda,k,0}(Q)\right)\geq \left|f_{\lambda,k,0}(P)-f_{\lambda,k,0}(Q)\right|\\
    &\geq \left|  f_{\lambda,k,0}(Q)-Q_{i} \right|-\left|Q_{i}-f_{\lambda,k,0}\left(Q'\right) \right|-\left| f_{\lambda,k,0}(P)-f_{\lambda,k,0}\left(Q'\right) \right|\\
    &\geq \frac{C^{**}}{2}-r_{0}-\text{dist}_{f_{\lambda, k,0}\left(\hat L_\lambda \right)}\left(f_{\lambda,k,0}(P),f_{\lambda,k,0}\left(P'\right)\right)\\
    &\geq \frac{C^{**}}{4}-\text{dist}_{f_{\lambda, k,0}\left(\hat L_\lambda \right)}\left(f_{\lambda,k,0}(P),f_{\lambda,k,0}\left(P'\right)\right).
\end{align*}
where the last inequality follows from (\ref{fixed inner ball radius}).
Therefore, since the segment $J$ is contained in $\mathbb B_i$, by Lemma \ref{bounded distortion} and (2) of Lemma \ref{inner rectangle construction}, we have
\begin{align*}
   &\frac{\text{length}(B)}{\text{length}(U)} =\frac{\text{dist}_{\hat L_\lambda}\left(P,P''\right)}{\text{dist}_{\hat L_\lambda}\left(P,P'\right)}\geq \frac{\text{dist}_{\hat L_\lambda}\left(P,Q\right)}{\text{dist}_{\hat L_\lambda}\left(P,P'\right)}\\
    &\geq C_4^{-1}\frac{\text{dist}_{f_{\lambda, k,0}\left(\hat L_\lambda \right)}\left(f_{\lambda,k,0}(P), f_{\lambda,k,0}\left(Q\right)\right)}{\text{dist}_{f_{\lambda, k,0}\left(\hat L_\lambda \right)}\left(f_{\lambda,k,0}(P), f_{\lambda,k,0}\left(P'\right)\right)}\geq C_4^{-1}\frac{\frac{C^{**}}{4}}{C_1 C^*\left(\mu_1-\delta\right)^{-\emb{N^*}}}-C_4^{-1}\\
    &\geq M,
\end{align*}
where the last inequality follows from (\ref{N large enough}). We then see that for each $\lambda\in[0,\lambda_0]$, we have 
$$\tau\left(C(\lambda, N^*)\right)\geq M$$ and from (\ref{Thickness and Hausdorff dimension}), we have
$$\text{dim}_H\left(C(\lambda, N^*)\right)\geq \frac{\log 2}{\log\left(2+\frac{1}{M}\right)}.$$
From (\ref{Cantor set containment}), this then implies (\ref{Hausdorff dimension of torus Cantor}) holds, for all $\lambda\in[0,\lambda_0]$. 
 \end{proof}

\section{\emb{Concluding the Proofs}}\label{Concluding the Proofs}

\begin{proof}[\emb{Proof of Theorem \ref{first intermediate theorem}.}]
    \emb{Fix a bounded sequence $\overline a=(a_n)$, $a_n\in\mathbb N$, and denote $\mathcal A:=\{a_n\}$. Let $\epsilon>0$ and suppose that $M>0$ is so large that
    $$\frac{\log 2}{\log\left(2+\frac{1}{M}\right)}>1-\epsilon.$$
    By Theorem \ref{intermediate theorem}, there is a $\rho>0$ and $\lambda_0\in\left(0, \min_{a\in\mathcal A}\left\{\tilde \lambda(\rho, a)\right\}\right)$ such that $\mathcal C\left(\lambda, \tilde O_{r(\rho)}\right)$ has Hausdorff dimension larger than $\frac{\log 2}{\log\left(2+\frac{1}{M}\right)}$.\newline 
\indent Set $\lambda\in(0,\lambda_0)$. Then, since $\lambda_0<\min_{a\in\mathcal A}\left\{\tilde \lambda(\rho, a)\right\}$, from Lemma \ref{connection}, we know that $\Lambda_{\lambda}^{\text{surv}}\left(\overline a\right)\cap L_{\lambda}$  is the image of $\mathcal C\left(\lambda, \tilde O_{r(\rho)}\right)$ under the map $F_{\lambda}$. Since $F$ is a local diffeomorphism, it follows that 
\begin{align*}
\text{dim}_H\left(\Lambda_{\lambda}^{\text{surv}}\left(\overline a\right)\cap L_{\lambda}\right)= \text{dim}_H\left( \mathcal C\left(\lambda, \tilde O_{r(\rho)}\right)\right)\geq 1-\epsilon. 
\end{align*}}
\end{proof}

\begin{proof}[\emb{Proof of Theorem \ref{main theorem}.}]
\emb{Let $\alpha=[a_1,a_2,\dots]\in (0,1)$ be an irrational of bounded type. Set $\overline a=(a_n).$ Let $\epsilon>0$. By Theorem \ref{first intermediate theorem}, there is a $\rho>0$ and $\lambda_0>0$ such that for any $\lambda\in[0,\lambda_0]$, we have
$$\text{dim}_H\left(\Lambda_{\lambda}^{\text{surv}}\left(\overline a\right)\cap L_{\lambda}\right)\geq 1-\epsilon.$$
Let $\lambda\in[0,\lambda_0]$. From Theorem \ref{trace map bounded orbits connection}, and since $\Lambda^{\text{surv}}_{\lambda}\left(\overline a, O_{\rho}\right)\subset \Lambda_{\lambda}^{\text{bnd}}\left(\overline a \right)$, we have 
$$\text{dim}_H\left(\sigma_{\lambda,\alpha}\right)=\text{dim}_H\left(  L_{\lambda}\cap  \Lambda_{\lambda}^{\text{bnd}}\left( \overline a \right)\right)\geq \text{dim}_H\left(\Lambda_{\lambda}^{\text{surv}}\left(\overline a\right)\cap L_{\lambda}\right)\geq 1-\epsilon.$$}
\end{proof}

\emb{
\begin{remark}
   We expect that Theorem \ref{main theorem} should hold for all irrational $\alpha\in(0,1)$. We would like to point out exactly where the assumption that $\alpha$ is of bounded type is used in our proofs. It is used in Proposition \ref{Circle rotation}, which is the main ingredient of Lemma \ref{estimate for irrational lines}, and this is what later allows us to control the distance of each piece of stable manifold $W^s_{\lambda}(Q_0, m)$ from each $Q_i$ in the construction of the rectangles in Subsection \ref{Particular Rectangle Construction}. The assumption also implies that the collection of maps $\left\{\hat T_{\lambda, a} : \lambda\in[0,\lambda_0],\  a\in\mathcal A\right\}$ from Section \ref{Main Estimates and Non-Stationary Trace Map Dynamics} have uniformly bounded $C^2$ norms, which is required for the bounded distortion estimate (see Lemma \ref{bounded distortion}). It also allows us to control the $C^2$-closeness of $\hat T_{\lambda,a}$ and $\hat T_{0,a}$ in a uniform way, which is used in many arguments such as Theorem \ref{stable manifold theorem} and Corollary \ref{perturbations of stable manifolds}.
\end{remark}}
\appendix

\section{\emb{Contraction Mappings}}
The following proposition is a non-stationary version of the contraction mapping principle.
\begin{propsec}\label{nonstationary contractions}
       Let $(X,d)$ be a complete metric space and $K\subset X$ a compact set. Suppose that there is a $\mu\in(0,1)$ and, for each $n\in\mathbb{N}$, a map $f_n:K\rightarrow K$ such that 
       $$d\left(f_n(x), f_n(y)\right)\leq \mu d(x,y),$$
       for all $x,y\in K$. Then, the following hold:

       \begin{itemize}
           \item[(1)] There is a $x^*\in K$ such that 
       $$\lim\limits_{n\rightarrow \infty}f_1\circ\cdots \circ f_n(x)=x^*,$$
       for any $x\in K$.
       \item[(2)] 
       
       For each $\epsilon>0$, there is a $\epsilon'>0$ such that if for each $n\in\mathbb N$, there is a mapping $h_n: K\rightarrow K$ with
     
       $$d\left(h_n(x), h_n(y)\right)\leq \mu' d(x,y) \ \text{and} \ \sup_{n}d_{C^0(K)}(\emb{f_n}, h_n)<\epsilon',$$
       for some $\mu'\in(0,1)$ and for all $x,y\in K$, then
       $$d\left(x^*, y^*\right)<\epsilon,$$
       where $y^*$ is the element such that 
       $$\lim\limits_{n\rightarrow \infty}h_1\circ\cdots \circ h_n(x)=y^*,$$
       for each $x\in X$.
       
       \end{itemize}
       
   \end{propsec}

   \begin{proof}
       \begin{itemize}
           \item[(1)] \emb{The proof here is analogous to that of the Banach Fixed Point Theorem. For a proof of this exact result, see Theorem 3 in \cite{LY}.}
           \item[(2)] Set $\epsilon'=\frac{\epsilon(1-\mu)}{2}$. Then, an induction argument gives that 
           $$d\left(f_1\circ\cdots\circ f_n(x), h_1\circ\cdots\circ h_n(x)\right)\leq \sum_{k=1}^{n}\epsilon'\mu^{k-1},$$
           for any $x\in K$.
          We can deduce that 
          $$d\left(x^*,y^*\right)\leq  d\left( f_1\circ\cdots\circ f_n(x), x^*\right)+d\left(y^*, h_1\circ\cdots\circ h_n(x)\right)+\sum_{k=1}^{n}\epsilon'\mu^{k-1}$$
          where the last quantity approaches $\frac{\epsilon'}{1-\mu}=\frac{\epsilon}{2}$ as $n\rightarrow\infty$.
       \end{itemize}
   \end{proof}

   \section{\emb{Continued Fractions}}

Given an irrational $\alpha=[a_1,a_2,\dots]\in(0,1)$, set $p_n(\alpha)\in \mathbb{Z}$ and $q_n(\alpha)\in\mathbb{N}$ to be the coprime integers such that
$$\frac{p_n(\alpha)}{q_n(\alpha)}=[a_1,a_2,\dots, a_n].$$
We know that these values satisfy the following recurrence equations:

\begin{equation}
    p_{k+1}(\alpha)=a_{k+1}p_{k}(\alpha)+p_{k-1}(\alpha), \ p_1(\alpha)=p_0(\alpha)=1
\end{equation}
\begin{equation}\label{recurrence relations}
    q_{k+1}(\alpha)=a_{k+1}q_{k}(\alpha)+q_{k-1}(\alpha), \ q_1(\alpha)=a_1, \ q_0(\alpha)=1
\end{equation}
and that 
\begin{equation}\label{Diophantine Approximation}
    \left|q_k(\alpha)\alpha-p_k(\alpha) \right|<\frac{1}{q_k(\alpha)}.
\end{equation}

In regards to an irrational rotation by $\alpha=[a_1, a_2,\dots]$ on $S^1$, the Three Distance Theorem (see for example Section 2.1 of \cite{I}) states that for each positive integer $n$, the values $\{m\alpha\}$, $1\leq m \leq n$, divide the unit circle into $n+1$ many intervals of at most three possible lengths, one of which is the sum of the other two. The precise lengths of the intervals come from how $n$ relates to the values $p_k(\alpha)$ and $q_k(\alpha)$. In particular, if we choose $n=(r+1)q_k(\alpha)+q_{k-1}(\alpha)-1$ for some $k\geq 1$ and $1\leq r<a_{k+1}$, then the set of points $\{m\alpha\}$, $1\leq m \leq n$, split $S^1$ into $n+1$ many intervals of only two possible lengths given by 
$$\left|q_{k}(\alpha)\alpha-p_k(\alpha)\right| \ \text{or} \ \left|q_{k-1}(\alpha)\alpha-p_{k-1}(\alpha)\right|-r\left|q_{k}(\alpha)\alpha-p_k(\alpha)\right|,$$
and by (\ref{Diophantine Approximation}), both of these quantities are no greater than $\frac{1}{q_{k-1}(\alpha)}+\frac{1}{q_{k}(\alpha)}$. As a corollary, we obtain the following result which will be used later on.

\begin{propsec}\label{Circle rotation}
Suppose that $\mathcal A\subset \mathbb{N}$ is finite. For each $\epsilon>0$, there is a $n=n(\epsilon, \mathcal A)\in\mathbb{N}$ such that for any $\alpha=[a_1,a_2,\dots]$ with $a_i\in\mathcal A$, we must have 
    $$\min\limits_{1\leq m \leq n}\left| m\alpha-y\right|<\epsilon$$
    for any $y\in S^1$.

\end{propsec}

\begin{proof} Let $\epsilon>0.$ Set $Y:=\{\alpha=[a_1,a_2,\dots] : a_i\in \mathcal A \}\subset [0,1]\setminus \mathbb{Q}$ and let $a'$ be the smallest element of $\mathcal A$. Then using the recurrence relation (\ref{recurrence relations}) and an induction argument, we find that for each $k\geq 1$, there is an $M_k>0$ such that 
$$q_k\left( \left[ \overline{a'}\right]\right)\leq\sup\limits_{\alpha\in Y}|q_k(\alpha)|\leq M_k.$$
Choose $k$ large enough so that 
$$\frac{1}{q_{k-1}\left( \left[ \overline{a'}\right]\right)}+\frac{1}{q_k\left( \left[ \overline{a'}\right]\right)}<\epsilon.$$
Set $n:=2M_k+M_{k-1}-1$ and let $\alpha\in Y$ be arbitrary. By the above discussion, for $n_k:=2q_{k}(\alpha)+q_{k-1}(\alpha)-1\leq n$, we have that the points $\{m\alpha\}$, $1\leq m <n_k$, split $S^1$ into $n_k+1$ many intervals of lengths no greater than 
$\frac{1}{q_{k-1}(\alpha)}+\frac{1}{q_{k}(\alpha)}$ which is less than $\epsilon$. So, it follows that the set $\{m\alpha\}_{m=1}^{n}$ splits $S^1$ into finitely many intervals of length less than $\epsilon$.

\end{proof}

\section*{Acknowledgements}
I would like to thank Anton Gorodetski for his guidance and support, \emb{and Jake Fillman for looking at a preliminary draft of this text}. \emb{I would also like to thank both anonymous referees for their helpful suggestions and input.} This project was supported by NSF grant DMS-2247966 (PI: A. Gorodetski).

\newcommand{\Addresses}{{
  \bigskip
  \footnotesize

  \textsc{Department of Mathematics, University of California, Irvine, CA 92697, USA}\par\nopagebreak
  \textit{E-mail address}: \texttt{lunaar1@uci.edu}

}}



\Addresses

\begin{thebibliography}{99}


\bibitem[BGJ]{BGJ}  Baake M., Grimm U., Joseph D.,  Trace maps, invariants,
and some of their applications, \textit{Internat.\ J.\ Modern
Phys.~B} \textbf{7} (1993), pp. 1527--1550.

\bibitem[BR]{BR}  Baake M., Roberts J., The dynamics of trace maps, in
\textit{Hamiltonian Mechanics} (\textit{Toru\'n, 1993}), pp. 275--285,
NATO Adv.\ Sci.\ Inst.\ Ser.~B Phys.\ \textbf{331}, Plenum, New
York, 1994.

\bibitem[BBL]{BBL} Band R., Beckus S., Loewy R., The Dry Ten Martini Problem for Sturmian Hamiltonians, (2024). arXiv:2402.16703 



\bibitem[BIST]{bist} Bellissard J.,  Iochum B., Scoppola E., and  Testard D., Spectral properties of one-dimensional
quasicrystals, {\it Commun. Math. Phys.}, vol. 125 (1989), pp. 527--543.

\bibitem[BM]{BM} Borissov G., Monakov G., Generalized Bounded Distortion Property. J Dyn Diff Equat (2024). https://doi.org/10.1007/s10884-024-10376-5

\bibitem[C]{C} Casdagli M., Symbolic dynamics for the renormalization map
of a quasiperiodic Schr\"odinger equation, \textit{Comm.\ Math.\
Phys.}\ \textbf{107} (1986), 295--318.

\bibitem[Ca]{Ca} Cantat S., Bers and H\'enon, Painlev\'e and Schr\"odinger, \textit{Duke Math.\ J.}\ \textbf{149} (2009), 411--460.

\bibitem[CQ]{CQ} Cao J., Qu Y., Almost sure dimensional properties for the spectrum and the density of states of Sturmian Hamiltonians, (2023). arXiv:2310.07305



\bibitem[D1]{D00} Damanik D., Gordon-type arguments in the spectral theory
of one-dimensional quasicrystals, in \textit{Directions in
Mathematical Quasicrystals}, 277--305, CRM Monogr.\ Ser.\
\textbf{13}, Amer. Math. Soc., Providence, RI, 2000.

\bibitem[D2]{D07} Damanik D., Strictly ergodic subshifts and associated
operators, in \textit{Spectral Theory and Mathematical Physics: a
Festschrift in Honor of Barry Simon's 60th Birthday},  505--538,
Proc.\ Sympos.\ Pure Math.\ \textbf{76}, Part 2, Amer. Math. Soc.,
Providence, RI, 2007.

\bibitem[D3]{D2000} Damanik D., Substitution Hamiltonians with bounded trace map orbits, {\it J. Math. Anal. Appl.}, vol. 249 (2000), no. 2, pp. 393--411.

\bibitem[D]{D} Damanik D., Schrödinger operators with dynamically defined potentials. \textit{Ergodic Theory and Dynamical Systems}, 37(6) (2017)., 1681-1764. doi:10.1017/etds.2015.120

\bibitem[DEG]{DEG} Damanik D., Embree M., Gorodetski A. (2012). Spectral properties of Schrödinger operators arising in the study of quasicrystals. arXiv: Mathematical Physics, 307-370.

\bibitem[DEGT]{degt} Damanik D., Embree M., Gorodetski A., Tcheremchantsev S.,
The Fractal Dimension of the Spectrum of the Fibonacci Hamiltonian, {\it Communications in Mathematical Physics}, vol. 280 (2008), no. 2, pp. 499-516.

\bibitem[DF1]{DF} Damanik D., Fillman J.,
One-Dimensional Ergodic Schrödinger Operators, I: General Theory. 
Graduate Studies in Mathematics, 221. American Mathematical Society, 2022.
ISBN-13: 978-1-4704-5606-1

\bibitem[DF2]{DF2} Damanik D., Fillman J.,
One-Dimensional Ergodic Schrödinger Operators, II: Specific Classes.
Graduate Studies in Mathematics, 249. American Mathematical Society, 2024.
ISBN-13: 978-1-4704-6503-2

\bibitem[DG1]{DG1} Damanik D., Gorodetski A., Hyperbolicity of the Trace Map for the Weakly Coupled Fibonacci Hamiltonian,  {\it Nonlinearity,}  vol. 22 (2009), pp. 123--143.

\bibitem[DG2]{dg2} Damanik D., Gorodetski A., The Spectrum of the Weakly Coupled Fibonacci Hamiltonian,
{\it Electronic Research Announcements in Mathematical Sciences}, vol. 16 (2009), pp. 23--29.

\bibitem[DG3]{dg3} Damanik D., Gorodetski A., Spectral and Quantum Dynamical Properties of the Weakly Coupled Fibonacci Hamiltonian, {\it  Communications in Mathematical Physics}, vol. 305 (2011), pp. 221--277.

\bibitem[DG4]{dg4} Damanik D., Gorodetski A., The Spectrum and the Spectral
Type of the Off-Diagonal Fibonacci Operator, arXiv:0807.3024v1

\bibitem[DG5]{dg5} Damanik D., Gorodetski A. The Density of States Measure of the Weakly Coupled Fibonacci Hamiltonian. Geom. Funct. Anal. 22, 976–989 (2012). https://doi.org/10.1007/s00039-012-0173-8

\bibitem[DGY]{DGY} Damanik D., Gorodetski A., Yessen W.N. (2014). \textit{The Fibonacci Hamiltonian}. Inventiones mathematicae, 206, 629-692.

\bibitem[FLW]{FLW} Fan S., Liu Q., Wen Z. (2011). Gibbs-like measure for spectrum of a class of quasi-crystals. \textit{Ergodic Theory and Dynamical Systems}, 31(6), 1669-1695. doi:10.1017/S0143385710000635



\bibitem[I]{I} Isola S., Continued Fractions and Dynamics, \textit{Applied Mathematics}, 5 (2014), 1067-1090. doi: 10.4236/am.2014.57101.

\bibitem[JL]{JL} Jitomirskaya S.Y., Last Y., Dimensional Hausdorff Properties of Singular Continuous Spectra, \textit{Phys. Rev. Lett.}, 76 (1996), 1765-1769. doi: https://doi.org/10.1103/PhysRevLett.76.1765


\bibitem[KH]{KH}  Katok A., Hasselblat B., Introduction to the Modern Theory of Dynamical System, Cambridge Univ. Press, 1995.

\bibitem[KKT]{KKT} Kohmoto M., Kadanoff L. P.,  Tang C., Localization
problem in one dimension: mapping and escape, \textit{Phys.\ Rev.\
Lett.}\ \textbf{50} (1983), pp. 1870-1872.

\bibitem[L]{L} Last Y., Quantum dynamics and decompositions of singular continuous spectra. \textit{Jounral of Functional Analysis}. \textbf{142} (1996), 406-445.
\bibitem[LW]{LW} Liu, QH., Wen, ZY. Hausdorff Dimension of Spectrum of One-Dimensional Schrödinger Operator with Sturmian Potentials. \textit{Potential Analysis}. 20, 33–59 (2004). https://doi.org/10.1023/A:1025537823884

\bibitem[LPW]{LPW}Liu, Qing-Hui; Peyrière, Jacques; Wen, Zhi-Ying. Dimension of the spectrum of one-dimensional discrete Schrödinger operators with Sturmian potentials. Comptes Rendus. Mathématique, Volume 345 (2007) no. 12, pp. 667-672. doi : 10.1016/j.crma.2007.10.048. http://www.numdam.org/articles/10.1016/j.crma.2007.10.048/

\bibitem[LQW]{LQW}
Qing-Hui Liu, Yan-Hui Qu, Zhi-Ying Wen,
The fractal dimensions of the spectrum of Sturm Hamiltonian,
\textit{Advances in Mathematics},
Volume 257,
2014,
Pages 285-336,
ISSN 0001-8708,
https://doi.org/10.1016/j.aim.2014.02.019.

\bibitem[LY]{LY} Luna A., Yang W., Generalized Fiber Contraction Mapping Principle, (2024). 	arXiv:2412.09767

\bibitem[M]{M} Mei M., Spectral properties of discrete Schr\"odinger operators with potentials generated by primitive invertible substitutions, {\it J. Math. Phys}. 55, 082701 (2014).


\bibitem[PT]{pt2} Palis J., Takens F., Hyperbolicity and sensitive
chaotic dynamics at homoclinic bifurcations. \textit{Cambridge University
Press}, 1993.

\bibitem[P]{P}  Pollicott M., Analyticity of dimensions for hyperbolic surface diffeomorphisms, \textit{Proc. Amer.
Math. Soc.} 143 (2015), 3465–3474
\bibitem[Ra]{Ra}  Raymond L., A constructive gap labelling for the discrete Schr\"odinger operator on a quasiperiodic chain, Preprint (1997).


\bibitem[Su1]{S87} S\"ut\H{o} A., The spectrum of a quasiperiodic
Schr\"odinger operator, \textit{Commun.\ Math.\ Phys.}\
\textbf{111} (1987), 409--415.

\bibitem[Su2]{S89} S\"ut\H{o} A., Singular continuous spectrum on a Cantor
set of zero Lebesgue measure for the Fibonacci Hamiltonian,
\textit{J.\ Stat.\ Phys.}\ \textbf{56} (1989), 525--531.

\bibitem[Su3]{S95} S\"ut\H{o} A., Schr\"odinger difference equation with
deterministic ergodic potentials, in \textit{Beyond Quasicrystals}
(\textit{Les Houches, 1994}), 481--549, Springer, Berlin, 1995.

\end{thebibliography}
\end{document}